\documentclass[11pt]{amsart}

\usepackage[T1]{fontenc}
\usepackage[utf8]{inputenc}
\usepackage{textcomp}
\usepackage{lmodern}
\usepackage{microtype}
\usepackage{enumitem}

\usepackage{amssymb}
\usepackage{amsfonts}
\usepackage{amsthm}
\usepackage{amscd}

\usepackage{hyperref}

\usepackage{mathscinet}
\usepackage[giveninits=true,doi=false,url=false,sortcites,backend=biber]{biblatex}
\newtheorem{theorem}{Theorem}[section]

\newtheorem{lemma}[theorem]{Lemma}

\newtheorem{cor}[theorem]{Corollary}
\newtheorem{question}[theorem]{Question}

\theoremstyle{definition}
\newtheorem{definition}[theorem]{Definition}

 \def \dom{\operatorname{dom}}
 
\allowdisplaybreaks[2]

\def\RCAo{\mathsf{RCA_0}}
\def\WKLo{\mathsf{WKL_0}}
\newcommand\EMinf{\mathrm{f EM}_{<\infty}}
\newcommand\ome{\omega}

\def\A{\forall}
\def\N{\mathbb{N}}
\def\Ii{\mathrm{I}\Sigma_1}
\def\II{\mathrm{I}\Sigma^0_1}
\def\III{\mathrm{I}\Sigma^0_2}

\newcommand\fin{\mathrm{fin}}
\newcommand\EM{\mathrm{EM}}
\def\Cod{\mathrm{Cod}}

\begin{document}

\title{Erdos-Moser and $I\Sigma_2$}
\author{Henry Towsner and Keita Yokoyama}
\date{\today}
%\thanks{Partially supported by NSF grant DMS-1600263}
\thanks{The first author
is partially supported by NSF grant DMS-1600263.
The second author
is partially supported by
JSPS KAKENHI (grant numbers 16K17640 and 15H03634) and JSPS Core-to-Core Program
(A.~Advanced Research Networks).}
\address {Department of Mathematics, University of Pennsylvania, 209 South 33rd Street, Philadelphia, PA 19104-6395, USA}
\email{htowsner@math.upenn.edu}
\urladdr{\url{http://www.math.upenn.edu/~htowsner}}
\address {School of Information Science, Japan Advanced Institute of Science and Technology, 1-1 Asahidai, Nomi, Ishikawa 923-1292, JAPAN}
\email{y-keita@jaist.ac.jp}

\begin{abstract}
The first-order part of the Ramsey's Theorem for pairs with an arbitrary number of colors is known to be precisely $B\Sigma^0_3$.  We compare this to the known division of Ramsey's Theorem for pairs into the weaker principles, $\mathsf{EM}$ (the Erd\H{o}s-Moser principle) and $\mathsf{ADS}$ (the ascending-descending sequence principle): we show that the additional strength beyond $I\Sigma^0_2$ is entirely due to the arbitrary color analog of $\mathsf{ADS}$.

Specifically, we show that $\mathsf{ADS}$ for an arbitrary number of colors implies $B\Sigma^0_3$ while $\mathsf{EM}$ for an arbitrary number of colors is $\Pi^1_1$-conservative over $I\Sigma^0_2$ and it does not imply $I\Sigma^0_2$.
\end{abstract}

\maketitle

\section{Introduction}

One much-studied project in reverse mathematics is determining the precise strength of Ramsey's Theorem for Pairs ($\mathsf{RT}^2$) \cite{cholak:MR1825173,Hirst-PhD,MR3194495,MR2963024}.  Lerman, Solomon, and the first author showed \cite{Lerman2013Separating} that Ramsey's Theorem for Pairs with two colors, $\mathsf{RT}^2_2$, splits into two parts---the Erd\H{o}s-Moser principle $\mathsf{EM}$ \cite{MR3659408} and the Ascending-Descending Sequence principle $\mathsf{ADS}$ \cite{MR2298478}.  Their argument suggests that $\mathsf{ADS}$ captures the aspect of $\mathsf{RT}^2_2$ which requires the construction of a distinct solution for each color.

In this paper, we examine this division over an arbitrary number of colors---that is, the way the principle $\mathsf{RT}^2$ divides into arbitrary color analogs of $\mathsf{ADS}$ and $\mathsf{EM}$.  Our focus is identifying the first-order part of the theory, which can usually be measured by conservativity over the hierarchy of induction principles $I\Sigma^0_1<B\Sigma^0_2<I\Sigma^0_2<B\Sigma^0_3<\cdots$ \cite{MR606791}.  Slaman and the second author showed \cite{slaman_yokoyama} that $\mathsf{RT}^2$ is $\Pi^1_1$-conservative over $B\Sigma^0_3$; combined with Hirst's result \cite{Hirst-PhD} that $\mathsf{RT}^2+\mathsf{RCA}_0$ implies $B\Sigma^0_3$, this precisely characterizes the first-order consequences of $\mathsf{RT}^2$.

This result is unusual: it is known \cite{cholak:MR1825173} that $\mathsf{RT}^2$ divides into two principles, $\mathsf{SRT}^2$ and $\mathsf{COH}$ where the former has low$_2$ solutions and the latter principle is known to be conservative over the weaker principle $I\Sigma^0_2$ (the standard reference on these induction principles is \cite{MR1748522}).  Most constructions of a low$_2$ solution can be adapted to give conservation over $I\Sigma^0_2$, so it is striking that conservation for $\mathsf{SRT}^2$ cannot be strengthened to conservation over $I\Sigma^0_2$.  In this case, the issue is that when the low$_2$ construction is adapted to $\mathsf{RT}^2$ with a potentially nonstandard number of colors, the construction requires attempts to build solutions in each color.  $B\Sigma^0_3$ is needed to show that if, at every step, the solution in some color gets extended, then there must be a single color whose solution is extended unboundedly many times.

This suggests that the generalization of $\mathsf{EM}$ to an arbitrary number of colors should be conservative over $I\Sigma^0_2$, while the generalization of $\mathsf{ADS}$ to an arbitrary number of colors should imply $B\Sigma^0_3$.  The goal of this paper is to confirm these guesses.

In order to do this, we need to choose appropriate generalizations of $\mathsf{EM}$ and $\mathsf{ADS}$ to an arbitrary number of colors.

One natural choice is to adapt the definitions given in \cite{MR2298478}:
\begin{definition}
  If $S$ is a set and $c:[S]^2\rightarrow[0,a]$ is a coloring of pairs from $S$, we say $c$ is \emph{transitive} if whenever $x,y,z\in S$, $x<y<z$, and $c(x,y)=c(y,z)$, also $c(x,y)=c(x,z)$.

$\mathsf{EM}_{<\infty}$ holds if, whenever $c:[\mathbb{M}]^2\rightarrow[0,a]$ is a finite coloring, there is an infinite subset $S\subseteq\mathbb{M}$ such that $c\upharpoonright [S]^{2}$ is transitive.

  $\mathsf{trRT}_{<\infty}$ holds if, whenever $c:[\mathbb{M}]^2\rightarrow[0,a]$ is a transitive coloring of pairs, there is an infinite set $S$ such that $c\upharpoonright[S]^2$ is constant.
\end{definition}
Since $\mathsf{trRT}_2$ (that is, $\mathsf{trRT}_{<\infty}$ restricted to the case where $a=1$) is equivalent to $\mathsf{ADS}$, this might seem like a natural notion.  However it is known to behave oddly---for example, it is not known whether $\mathsf{trRT}_{n}$ implies $\mathsf{trRT}_{n+1}$ for any $n$ \cite{MR2298478,MR3219047}, leading to the unusual situation where the strength of a principle may depend on the number of colors.

$\mathsf{EM}_{<\infty}$ is similarly odd; for example, it is not clear that $\mathsf{EM}_{<\infty}$ is enough to show that it is possible to find solutions to finitely many simultaneous instances of $\mathsf{EM}$.

The following notion seems to behave more naturally:
\begin{definition}
    If $S$ is a set and $c:[S]^2\rightarrow[0,a]$ is a coloring of pairs from $S$, we say $c$ is \emph{fallow} if whenever $x,y,z\in S$, $x<y<z$, $c(x,z)\in\{c(x,y),c(y,z)\}$.

$\mathsf{fEM}_{<\infty}$ holds if, whenever $c:[\mathbb{M}]^2\rightarrow[0,a]$ is a finite coloring, there is an infinite subset $S\subseteq\mathbb{M}$ such that $c\upharpoonright [S]^{2}$ is fallow.

  $\mathsf{fRT}_{<\infty}$ holds if, whenever $c:[\mathbb{M}]^2\rightarrow[0,a]$ is a fallow coloring of pairs, there is an infinite set $S$ such that $c\upharpoonright[S]^2$ is constant.
\end{definition}
When $a=1$, fallowness and transitivity are equivalent, but in general fallowness is a stricter requirement than transitivity, so $\mathsf{fEM}_{\infty}$ implies $\mathsf{EM}_{<\infty}$, while $\mathsf{trRT}_{<\infty}$ implies $\mathsf{fRT}_{<\infty}$.

As some evidence that fallowness behaves reasonably, observe that $\mathsf{ADS}$ implies $\mathsf{fRT}_n$ for all finite $n$ and that $\mathsf{fEM}_{<\infty}$ shows that it is possible to solve finitely many instances of $\mathsf{EM}$ simultaneously.  We verify this last fact.
\begin{definition}
$\mathsf{EM}_{\times}$ holds if, whenever $\{c_i\}_{i\leq a}$ is a finite collection of colorings $c_i:[\mathbb{M}]^2\rightarrow\{0,1\}$, there is an infinite set $S\subseteq\mathbb{N}$ such that $c_i\upharpoonright[S]^2$ is transitive.
\end{definition}

\begin{lemma}
$\mathsf{EM}_\times$ and $\mathsf{fEM}_{<\infty}$ are equivalent.
\end{lemma}
\begin{proof}
  Suppose $\mathsf{EM}_\times$ holds and let $c:[\mathbb{M}]^2\rightarrow[0,a]$ be given.  For each $i\leq a$, define
\[c_i(x,y)=\left\{\begin{array}{ll}
1&\text{if }c(x,y)=i\\
0&\text{otherwise}.
\end{array}\right..\]
By $\mathsf{EM}_\times$, we find an infinite set $S$ on which every $c_i$ is homogeneous.  We claim that $c$ is fallow on $S$.  Let $x<y<z$ be given.  Then $c_{c(x,z)}(x,z)=1$.  Therefore we cannot have $c_{c(x,z)}(x,y)=c_{c(x,z)}(y,z)=0$, since this would contradict the transitivity of $c_{c(x,z)}$, so either $c_{c(x,z)}(x,y)=1$ or $c_{c(x,z)}(y,z)=1$, and therefore $c(x,z)\in\{c(x,y),c(y,z)\}$.

Conversely, suppose $\mathsf{fEM}_{<\infty}$ holds and let $\{c_i\}_{i\leq a}$ be given.  We define $c:[\mathbb{M}]^2\rightarrow[0,2^{a+1}]$ by setting $c(x,y)=\sum_{i\leq a}2^ic_i(x,y)$.  By $\mathsf{fEM}_{<\infty}$, we obtain a set on which $c$ is fallow.  Let $x<y<z$ be given, and suppose $c_i(x,y)=c_i(y,z)$, so both $c(x,y)$ and $c(y,z)$ have $c_i(x,y)$ as their $i$-th bit.  Since $c(x,y)\in\{c_i(x,y),c_i(y,z)\}$, in particular the $i$-th bit of $c(x,z)$ is also $c_i(x,y)$, so $c_i(x,z)=c_i(x,y)$.
\end{proof}
Note that in the last step, we really appear to need fallowness of $c$---transitivity would not suffice---and it appears that $\mathsf{EM}_{<\infty}$ may be strictly weaker.

We can now state our results.
\begin{theorem}
$\mathsf{fRT}_{<\infty}+\mathsf{RCA}_0$ implies $B\Sigma^0_3$.
\end{theorem}

% \begin{theorem}
%   $\mathsf{fEM}_{<\infty}+\mathsf{RCA}_0$ implies $I\Sigma^0_2$.
% \end{theorem}

\begin{theorem}
  $\mathsf{fEM}_{<\infty}+I\Sigma^0_2$ is a $\Pi^1_1$-conservative extension of $I\Sigma^0_2$.
\end{theorem}

\begin{theorem}
  $\mathsf{WKL}_0+\mathsf{fEM}_{<\infty}$ is $\tilde\Pi^0_3$-conservative over $\mathsf{RCA}_0$, and thus it does not imply $I\Sigma^0_2$.
\end{theorem}

The reader who prefers transitivity to fallowness will be pleased to note that the analogous results hold for those notions as well: since $\mathsf{trRT}_{<\infty}$ implies $\mathsf{fRT}_{<\infty}$, we have $\mathsf{trRT}_{<\infty}+\mathsf{RCA}_0$ implies $B\Sigma^0_3$, and since $\mathsf{fEM}_{<\infty}$ implies $\mathsf{EM}_{<\infty}$, $\mathsf{EM}_{<\infty}+I\Sigma^0_2$ is a $\Pi^1_1$-conservative extension of $I\Sigma^0_2$.

\section{Implication}

%\subsection{$\mathsf{trRT}_{<\infty}$}
In this section we prove our first result, that $\mathsf{trRT}_{<\infty}+\mathsf{RCA}_0$ implies $B\Sigma^0_3$.

% \begin{lemma}[$\mathsf{RCA}_0+B\Sigma^0_2$]
% Fix a $\Sigma^0_0$-formula $\theta(x,y,z)$.  Then there is a function $h:\mathbb{M}\rightarrow\mathbb{M}$ such that $\forall z h(z)\leq z$ and
% \[\forall x(\exists y

% \end{lemma}

We can even work with a slightly weaker principle:
\begin{definition}
  A coloring $c:[\mathbb{M}]^2\rightarrow[0,a]$ is \emph{stable} if for every $x$ there is an $m$ so that for all $n\geq m$, $c(x,n)=c(x,m)$.

    $\mathsf{fSRT}_{<\infty}$ holds if, whenever $c:[\mathbb{M}]^2\rightarrow[0,a]$ is a stable fallow coloring of pairs, there is an infinite set $S$ such that $c\upharpoonright[S]^2$ is constant.
\end{definition}

%$\forall x\leq a\exists y\theta(x,y)\rightarrow\exists z\forall x\leq a\exists y\leq z\theta(x,y)$

\begin{lemma}[$\mathsf{RCA}_0+B\Sigma^0_2$]
  $B\Sigma^0_3$ is equivalent to
  \begin{quote}
    for any $h:u\times\mathbb{M}\rightarrow\mathbb{M}$ such that $\forall x<u\,\exists m\,\exists^\infty z \ h(x,z)<m$, there is an $M\in\mathbb{M}$ such that $\forall x<u\,\exists^\infty z\ h(x,z)<M$.
  \end{quote}
\end{lemma}
\begin{proof}
  First, suppose $B\Sigma^0_3$ holds and let $h:u\times\mathbb{M}\rightarrow\mathbb{M}$ be given such that $\forall x<u\,\exists m\,\exists^\infty z \ h(x,z)<m$ be given---that is,
\[\forall x<u\,\exists m\,\forall n\,\exists z>n\ h(x,z)<m.\]
It is well-known \cite{MR1748522} that $B\Sigma^0_3$ is equivalent to $B\Pi^0_2$, so
\[\exists M\,\forall x<u\,\exists m\leq M\,\forall n\,\exists z>n\ h(x,z)<m.\]
This implies
\[\exists M\,\forall x<u\,\forall n\,\exists z>n\ h(x,z)<M.\]

Conversely, let $\phi(a,b)$ be a $\Pi^0_2$ statement, $\forall y\exists z\theta(a,b,y,z)$.   Define a function
\[q_{a}(b,z)=\min\{y\leq z\mid \forall b'\leq b\exists y'\leq y\forall z'\leq z\ \neg\theta(a,b',y',z')\}\]
and $z$ if there is no such $y$.  We then define $r_{a}(z)$ to be the least $b\leq z$ so that $q_{a}(b,z)<q_{a}(b,z+1)$ if there is such a $b$, and $z$ otherwise.

We claim that $\forall b\exists y\forall z\neg\theta(a,b,y,z)$ iff $\lim_{z\rightarrow\infty}r_{a}(z)=\infty$.  Suppose $\forall b\exists y\forall z\neg\theta(a,b,y,z)$ holds.  Then, for each $b$, we may choose $Y$ (using $B\Sigma^0_2$) so that $\forall b'\leq b\exists y\leq Y\forall Z\neg\theta(a,b',y,z)$.  Then for $b'\leq b$ and $z\leq Y$, we have $q_a(b',z)=q_a(b',z+1)$, and so $r_a(z)>b$ once $z\geq Y$.  Since, for each $b$, $r_a(z)\geq b$ for sufficiently large $z$, $\lim_{z\rightarrow\infty}r_{a,b}(z)=\infty$.

Conversely, suppose there is a $b$ so that $\forall y\exists z\theta(a,b,y,z)$ holds.  Then for any $z_0$, consider $y=q_{a}(b,z_0)$.  There must be some least $z$ so that $\theta(a,b,y,z)$ holds, and since $y=q_{a}(b,z_0)$, either $z>z_0$ or $y=z$; if $z>z_0$ then $q_{a}(b,z-1)<q_{a}(b,z)$, so $r_{a}(z-1)\leq b$.  If $z\leq z_0$ then $q_{a}(b,z_0+1)=z_0+1>y=q_{a}(b,z_0)$, so $r_{a}(z_0)\leq b$.  In either case, there is a $z\geq z_0$ with $r_{a}(z)\leq b$.  Since this holds for any $z_0$, $\lim_{z\rightarrow\infty}r_{a}(z)\neq\infty$.

Suppose that $\forall a<A\exists b \forall y\exists z\theta(a,b,y,z)$ holds and define $h(a,z)=r_a(z)$.  Then for any $a<A$ there is a $b$ so that $r_a(z)\leq b$ infinitely often, and therefore infinitely many $z$ so that $h(a,z)<b+1$.  Therefore, by assumption, there is a $B$ so that for every $a<A$, $h(a,z)<B$ infinitely often.  Therefore for each $a<A$, there is a $b<B$ so that $\forall y\exists z\theta(a,b,y,z)$.
\end{proof}

\begin{theorem}
  $\mathsf{fSRT}_{<\infty}+\mathsf{RCA}_0$ implies $B\Sigma^0_3$.
\end{theorem}
\begin{proof}
Let $\mathcal{M}$ be a model of $\mathsf{fSRT}_{<\infty}+\mathsf{RCA}_0$ and write $\mathbb{M}$ for the universe of the first-order part.

By \cite{MR2566574}, $B\Sigma^0_2$ holds in $\mathbb{M}$, so we may use the equivalent formulation of $B\Sigma^0_3$ given by the previous lemma.  Consider a function $h:[0,u]\times\mathbb{M}\rightarrow\mathbb{M}$ such that $\forall x<u\,\exists m\,\exists^\infty z\ h(x,y)<m$.  

We define a function $f:[\mathbb{M}]^2\rightarrow[0,u]$ by defining
\[f(a,b)=\max\{x<u\mid \forall x'<u(\min_{z\in[a,b)}h(x',z)\leq \min_{z\in [a,b)}h(x,z))\}.\]
That is, for each $x$, we may let $q_{a,b}(x)=\min_{z\in[a,b)}h(x,z)$, and we choose $f(a,b)$ to be a value of $x$ maximizing $q_{a,b}$.  Since there might be multiple such values of $x$, we choose the largest one (arbitrarily).

$f$ is a $u$-coloring, and is stable and fallow.  It is easy to see that $f$ is fallow: for any $a<b<c$, $q_{a,c}(x)=\min\{q_{a,b}(x),q_{b,c}(x)\}$, so $x$ maximizes both $q_{a,b}$ and $q_{b,c}$ iff $x$ maximizes $q_{a,c}$ as well.

To see that $f$ is stable, fix any $a\in\mathbb{M}$.  For each $x$, there is some minimum value $\min_{[a,\infty)}h(x,z)$, and some $d$ by which this minimum is achieved: $\forall x<u\exists d \forall b\geq d q_{a,b}(x)=q_{a,d}(x)$.  Therefore, by $B\Sigma^0_2$, there is some $D$ so that for all $x<u$, $\forall b\geq D q_{a,b}(x)=q_{a,D}(x)$.  Therefore $f(a,b)=d(a,D)$ for all $b\geq D$.

By $\mathsf{fSRT}^2$, there is an infinite set $S$ so that $f\upharpoonright[S]^2$ is constant.  Let $x_0<u$ be the color which $f$ is constantly equal to on $S$.  By the assumption on $h$, there is an $M\in\mathbb{M}$ so that, for infinitely many $z$, $h(x_0,z)<M$.  Therefore for each $s\in S$, there is a $z>s$ with $h(x_0,z)<M$, and therefore an $s'>z$ with $s'\in S$, so $f(s,s')=x_0$, and then for all $x<u$, $q_{s,s'}(x)\leq q_{s,s'}(x_0)<M$, so there is a $z>s$ with $h(x,z)<M$.
\end{proof}

\section{Conservativity}

In this section we prove that $\mathsf{fEM}_{<\infty}+\mathsf{RCA}_0$ is $\Pi^1_1$-conservative over $\mathsf{RCA}_0+I\Sigma^0_2$.  Again, it suffices to work with the stable version.

\begin{definition}
  $\mathsf{fSEM}_{<\infty}$ holds if whenever $c:[\mathbb{M}]^2\rightarrow[0,a]$ is a stable finite coloring, there is an infinite subset $S\subseteq\mathbb{M}$ such that $c\upharpoonright [S]^{2}$ is fallow.

When $c:[\mathbb{M}]^2\rightarrow[0,d]$ is stable, we write $c_\infty:\mathbb{M}\rightarrow[0,d]$ for the corresponding limit coloring.
\end{definition}

Because $\mathsf{fEM}_{<\infty}$ is equivalent to $\mathsf{fSEM}_{<\infty}$ together with $\mathsf{COH}$ and all these principles are $\Pi^1_2$ statements, by \cite{MR2798907} it suffices to prove that $\mathsf{fSEM}_{<\infty}$ is $\Pi^1_1$-conservative over $\mathsf{RCA}_0+I\Sigma^0_2$.

Throughout this section, we assume we are working in a model of $I\Sigma^0_2$; in particular, unless otherwise specified, all arguments by induction have an inductive statement which is $\Sigma_2$ or $\Pi_2$.  

For the remainder of the section, we fix a stable coloring $c:[\mathbb{M}]^2\rightarrow[0,d]$.

\subsection{Motivation}

Before giving the detailed construction, we outline the main idea.  In \cite{Lerman2013Separating}, a forcing notion was used to construct solutions to instances of the Erd\H{o}s-Moser principle while avoiding the construction of solutions to stable Ramsey's theorem for pairs.  That construction used Matthias conditions $(f,S)$ where $f$ is a finite initial segment of the generic and $S$ is a certain kind of tree of possible extensions (closely related to bounded monotone enumerations \cite{MR3194495,kreuzer:1109.4277,MR3518781}).

Specifically, $S$ is a function where, for each $n$, $S(n)$ is a finite set of subsets where, when $m>n$, each set in $S(m)$ extends a set in $S(n)$.  When we extend $f$, we choose extensions from some subset in $S(n)$.  Not every element of $S(n)$ needs to have an extension in $S(m)$---some elements of $S(n)$ may represent dead ends. Crucially, there is an ``up or out'' property: we are promised that, in $S(m)$, every element of $S(n)$ is either \emph{properly} extended or has no extensions.  (The bad case would be if a finite set stuck around forever without extending: we need to be able to identify dead ends at some finite stage.)

The families that actually occur all have the same form---we partition an interval into components with the property that certain computations do not halt on any extension taking all its elements from a single component.  So when we wish to extend $f$, we take one of these partitions $[s,s+n]=\bigcup_{j<r}g_j$, pick some $j<r$ and some $g'\subseteq g_j$ witnessing a computation, and extend $f$ to $f\cup g'$.  We need to ensure that $\bigcup_{j<r}g_j$ extends to a partition of $[s,\infty)$ (that is, to a partition belonging to our family---one in which every component restricts computations as needed), which is a $\Pi_1$ property, but we also need to ensure that the $j$-th component is infinite in all such partitions.  In order to make the conservation argument, we want the property ``$f\cup g'$ is a possible extension'' to be a $\Sigma_2$ property, but requiring that the $j$-th component is guaranteed to be infinite is a $\Pi_2$ question (saying that we find extensions infinitely often).

To address this, we need to examine the behavior of the construction a bit more carefully.  Suppose we are given a valid partition $[s,s+n]=\bigcup_{j<r}g_j$, and it satisfies the $\Pi_1$ property that, for every $m>n$, there is a valid partition $[s,s+m]=\bigcup_{j<r}h_j$ such that, for each $j<r$, $g_j\subseteq h_j$.  (By a ``valid'' partition, we mean one imposing a suitable restraint on computations.)  The problem is that there might be some $m>n$ and some valid partition $[s,s+m]=\bigcup_{j<r}h_j$ so that $g_j\subseteq h_j$ for each $j<r$, but $\bigcup_{j<r}h_j$ represents a dead end: when $m'>m$, there exists some valid partition, $[s,s+m']=\bigcup_{j<r}h'_j$ with $h_j\subseteq h'_j$ for each $j<r$, but in any such partition, there is some $j<r$ with $h_j=h'_j$.

If we are given the partition $\bigcup_{j<r}h_j$, though, we can verify this fact in a $\Pi_1$ way---that is, the property that there are extensions of this partition for every $m'>m$ is $\Pi_1$, as is the property that there are no extensions in which $h_j$ is properly extended.

More generally, there might be several $h_j$ which are dead ends---that is, there might be an $R\subseteq[0,r)$ so that in any extension to a partition $\bigcup_{j<r}h'_j$, $h_j=h'_j$ for all $j\in R$ at once.

Given $\{h_j\}_{j<r},R$, we can verify this property in a $\Pi_1$ way, but we cannot check that $R$ is maximal---that is, there could still be some $j\in[0,r)\setminus R$ which will cause problems.  

Our solution is to say that we find an extension of $f$ if we have $[s,s+m]=\{h_j\}_{j<r},R$ and, for \emph{each} $j\in[0,r)\setminus R$, and extension $f'_j\subseteq h_j$ so that:
\begin{itemize}
\item $f\cup f'_j$ witnesses some $\Sigma_1$ property (say, that some computation halts),
\item for every $m'>m$, there is a valid partition $[s,s+m']=\{h'_j\}_{j<r}$ with $h_j\subseteq h'_j$ for all $j<r$ and $h_j=h'_j$ for all $j\in R$.
\end{itemize}
Then the existence of an extension becomes a $\Sigma_2$ property: we can find extensions if there is some $\{h_j\}_{j<r},R$ which is valid, $\{h_j\}_{j<r}$ extends to a partition of $[s,\infty)$, and in every such partition, the $h_j$ for $j\in R$ do not extend, and we can find suitable extensions of $f$ in every $h_j$ with $j\in[0,r)\setminus R$.  We cannot verify that $R$ is maximal---there might be other branches in $[0,r)\setminus R$ which are dead ends---but it does no harm to find possible extensions of $f$ in dead ends which we later discard.  What we are promising is that for each $j$, \emph{either} $j$ is a dead end \emph{or} we can find the witness we need in $h_j$, without worrying about the fact that there may be an overlap between these cases.

We end up needing to create a tree of extensions, to keep track of all the possible ways $f$ might extend in different branches.  However we will be able to control this branching by thinning it out unboundedly often.  This will let us show that the tree has a unique infinite branch, and we will arrange for the unique branch of this tree to be the set on which the coloring is transitive and satisfy $I\Sigma_2$.

\subsection{Families of Partitions}

The restraint in our Matthias conditions will be a ``family of partitions''; we now define this notion and what it means for one family of partitions to refine another.

\begin{definition}
A \emph{family of partitions of size $r$} is a function $S(n)$ such that, for some $s=\min S$:
\begin{itemize}
\item for each $n$, $S(n)$ is a set of partitions of $[s,s+n]$ into $r$ sets,
\item whenever $\{f_j\}_{j< r}\in S(n)$ and $m<n$, $\{f_j\cap[s,s+m]\}_{j<r}\in S(m)$.
\end{itemize}

We say $\{f_j\}_{j<r}\subseteq\{g_j\}_{j<r}$ if for all $j<r$, $f_j\subseteq g_j$.

%We say $\{g_j\}_{j\in R'}$ \emph{matches} $\{f_j\}_{j\in R}$ if $R\subseteq R'$ and for all $j\in R$, $f_j=g_j$.

% An \emph{extinction witness} for $S$ is a set $R\subseteq[0,r-1]$, an $n_0$, and $\{f_j\}_{j\in R}$ such that for every $n\geq n_0$, there is a partition $\{g_j\}_{j<r}\in S(n)$ matching $\{f_j\}_{j\in R}$.  We typically write $\{f_j\}_{j\in R}$ for the extinction witness, leaving $n_0$ implicit.

% An extinction witness $\{f_j\}_{j\in R}$ is \emph{maximal} if there does not exist an extinction witness $\{f'_j\}_{j\in R'}$ with $R\subsetneq R'$ and matching $\{f_j\}_{j\in R}$.
\end{definition}

%Note that being an extinction witness is a $\Pi_1$ property, while being a maximal extinction witness is a $\Pi_2$ property.  The significance of a maximal extinction witness is that when $\{f_j\}_{j\in R}$ is a maximal extinction witness and we find some $\{g_j\}_{j<r}\in S(n)$ matching $\{f_j\}_{j\in R}$, we can take any $j\in[0,r-1]\setminus R$, and there must be an $m>n$ so that for every $\{h_j\}_{j<r}\in S(m)$ with $\{g_j\}_{j<r}\subseteq\{h_j\}_{j<r}$, $g_j\subsetneq h_j$.  (Possibly because there are no such $\{h_j\}_{j<r}\in S(m)$.)

%In particular, this means that, for each $j_0\in [0,r-1]\setminus R$, the set of $g_{j_0}$ appearing in a partition matching $\{f_j\}_{j\in R}$ is roughly a bounded monotone enumeration as in \cite{?}.

% \begin{definition}
%     When $\{f_j\}_{j\in R}$ is an extinction witness, for each $j_0\in[0,r-1]\setminus R$ we define $\overline{S}_{j_0,\{f_j\}_{j\in R}}(n)$ to be the set of $g_{j_0}$ so that there is some $\{h_j\}\in S_j(n)$ with $h_j=f_j$ for $j\in R$ and $h_{j_0}=g_{j_0}$.
% \end{definition}

% The discussion above can be made precise.
% \begin{lemma}\label{thm:exts_exist}
% If $\{f_j\}_{j\in R}$ is a maximal extinction witness then for every $j_0\in[0,r-1]\setminus R$, and any $n$, there is an $m>n$ with $\overline{S}_{j_0,\{f_j\}_{j\in R}}(n)\cap\overline{S}_{j_0,\{f_j\}_{j\in R}}(m)=\emptyset$.
% \end{lemma}

\begin{definition}
  We say $S$ is \emph{infinite} if, for every $n$, $S(n)$ is non-empty.  We say $S$ is \emph{extensive} if whenever $\{g_j\}_{j<r}\in S(n)$, there is an $m>n$ so that for any $\{h_j\}_{j<r}\in S(m)$ with $\{g_j\}_{j<r}\subseteq \{h_j\}_{j<r}$, $g_j\neq h_j$ for all $j<r$.

When $\{g_j\}_{j<r}\in S(n)$, we say $\{g_j\}_{j<r}$ is \emph{permanent (in $S$)} if, for all $m\geq n$, there is an $\{h_j\}_{j<r}\in S(m)$ with $\{g_j\}_{j<r}\subseteq\{h_j\}_{j<r}$.
\end{definition}
Note that being infinite is a $\Pi_1$ property while being extensive is a $\Pi_2$ property.

\begin{lemma}
  If $\{g_j\}_{j<r}\in S(n)$ is permanent in $S$ then for every $m\geq n$ there is a permanent $\{h_j\}_{j<r}\in S(m)$ with $\{g_j\}_{j<r}\subseteq\{h_j\}_{j<r}$.
\end{lemma}
\begin{proof}
Let $m\geq n$ be given.  There are only finitely many $\{h_j\}_{j<r}\in S(m)$ with $\{g_j\}_{j<r}\subseteq\{h_j\}_{j<r}$, so we may choose $m'\geq m$ minimizing the number of $\{h_j\}_{j<r}\in S(m)$ such that there exists an $\{h'_j\}_{j<r}\in S(m')$ with $\{h_j\}_{j<r}\subseteq\{h'_j\}_{j<r}$.

There must be some $\{h'_j\}_{j<r}\in S(m')$ with $\{g_j\}_{j<r}\subseteq\{h'_j\}_{j<r}$, and we see that $\{h_j\}_{j<r}=\{h'_j\cap [\min S,\min S+m]\}_{j<r}$ must be permanent, because if there were any $m''\geq m'$ with no extension of $\{h_j\}_{j<r}$ then $m''$ would contradict the minimality in the choice of $m'$.
\end{proof}

\begin{definition}
  If $S$ is a family of partitions of size $r$ and $S'$ is a family of partitions of size $r'\geq r$, we say $S'$ \emph{refines} $S$ \emph{via} $\rho$ if $s\leq s'$, $\rho:[0,r')\rightarrow[0,r)$ and, for every $n$ and every $\{f_j\}_{j<r'}\in S'(n)$ there is a $\{g_j\}_{j<r}\in S(n+(s'-s))$ such that, for each $j<r$, $g_j\cap[s',s'+n]=\bigcup_{j'\in\rho^{-1}(j)}f_{j'}$.  We $S'$ \emph{surjectively refines} $S$ if $\rho$ is surjective.

%When $\{f_j\}_{j\in R}$ is a maximal extinction witness for $S$ then we say $S'$ refines $S$ \emph{surjectively up to $\{f_j\}_{j\in R}$} if $|\rho^{-1}(j)|=1$ for $j\in R$ and whenever $n\geq s'-s$ and $\{g_j\}_{j\leq r}\in S(n)$ matches $\{f_j\}_{j\in R}$, there is some $\{f'_{j'}\}_{j'<r'}\in S'(n)$ such that $\bigcup_{j'\in\rho^{-1}(j)}f'_{j'}=f_j$.
\end{definition}
It is easy to see that if $S''$ refines $S'$ via $\rho'$ and $S'$ refines $S$ via $\rho$ then $S''$ refines $S$ via $\rho\circ\rho'$.

\begin{lemma}\label{thm:partition_extensive}
  If $S$ is infinite then there is an extensive $S'$ refining $S$.
\end{lemma}
\begin{proof}
  Consider those $\{f_j\}_{j\in R}$ so that, for all sufficiently large $n$, there is a $\{g_j\}_{j<r}\in S(n)$ with $f_j=g_j$ for all $j\in R$.  The existence of such an $\{f_j\}_{j\in R}$ is a $\Sigma_2$ property, so we may choose $R$ maximal so that such an $\{f_j\}_{j\in R}$ exists, and then define $S'$ to be the family of partitions of size $r-R$ by choosing $\min S'$ to be larger than $\max_{j\in R}f_j$, choosing $\rho:[0,r-R-1]\rightarrow[0,r-1]$ so that the image of $\rho$ is $[0,r-1]\setminus R$ and placing $\{g_j\}_{j<r-R}\in S'(m)$ if there is some $\{h_j\}_{j<r}\in S(m)$ such that:
  \begin{itemize}
  \item for $j\in R$, $h_j=f_j$,
  \item for $j\not\in R$, $h_j\cap[s,m]=g_{\rho^{-1}(j)}$.
  \end{itemize}

The fact that $S'$ is infinite follows since, for all sufficiently large $n$, there is a $\{g_j\}_{j<r}\in S(n)$ with $f_j=g_j$ for all $j\in R$, so $\{g_{\rho(j)}\cap[s,n]\}_{j<r-R}\in S'(n)$, and the fact that $S'$ is extensive follows since $R$ is maximal.
\end{proof}

\begin{lemma}\label{thm:refine_partitions}
 Let $S$ be extensive and let $S^*$ be infinite with $\min S=\min S^*$.  Let $j<\min S$ be given and suppose that, for every $n$ there is a $\{g_j\}_{j<r}\in S(n)$ and an $\{h_{j^*}\}_{j^*<r^*}$ such that $\bigcup_{0<j^*<r^*}h_{j^*}=g_j$.  Then there is an extensive $S'$ which is a common refinement of $S$ and $S^*$ and a surjective refinement of $S$.
\end{lemma}
We are really asking about ways of partitioning $g_j$---we are interested in elements of $S^*$ where $h_0=\bigcup_{j'\neq j}g_{j'}$ and the other $r^*-1$ pieces give a partition of $g_j$.
\begin{proof}
  Choose $\rho_0:[0,r+r^*-2)\rightarrow[0,r)$ so that $|\rho_0^{-1}(j)|=r^*-1$ and $|\rho_0^{-1}(j')|=1$ for $j'\neq j$.  Fix $\rho_0^*:[0,r+r^*-2)\rightarrow[0,r^*)$ surjective so that $\rho_0^*(j')=0$ iff $\rho_0(j')\neq j$.

Define $S'_0(n)$ to consist of those $\{g'_{j'}\}_{j'<r+r^*-2}$ such that there is some $\{g_j\}_{j<r}\in S(n)$ and some $\{h_{j^*}\}_{j^*<r^*}\in S^*(n)$ such that:
  \begin{itemize}
  \item if $\rho_0(j')\neq j$ then $g'_{j'}=g_{\rho_0(j')}$, and
  \item if $\rho_0(j')=j$ then $g'_{j'}=h_{\rho_0^*(j')}$.
  \end{itemize}

We then refine $S'_0$ to an extensive $S'$ as in the previous lemma; since $S$ is already extensive, we are assured that when we choose $\{f_{j'}\}_{j'\in R}$ as in the previous lemma, $R\subsetneq \rho_0^{-1}(j)$, so the refinement of $S$ is surjective.
\end{proof}

% \begin{definition}
%   We say a finite set $f$ is \emph{acceptable} if $c$ is transitive on $f$ and whenever $a,b\in f$ and $c(a,b)=c_\infty(b)$, also $c_\infty(a)=c_\infty(b)$.
% \end{definition}

%In order to construct the desired set on which $c$ is transitive, we will need an intermediate step.

\subsection{Conditions}

\begin{definition}
  By a \emph{tree of finite sets}, we mean a function $\mathcal{F}$ such that:
  \begin{itemize}
  \item for each $n$, $\mathcal{F}(n)$ is a finite collection of finite sets,
  \item each $F(n)$ is non-empty.
  %\item if $f\in\mathcal{F}(n)$ then no proper subset of $f$ is in $\mathcal{F}(n)$,
  \item for each $f\in\mathcal{F}(n+1)$, $f\cap[0,\max\mathcal{F}(n)]\in\mathcal{F}(n)$,
%  \item if $f\in\mathcal{F}(n)\cap\mathcal{F}(m)$ for some $m<n$ then $f\in\mathcal{F}(k)$ for all $k\geq m$.
  \item for each $f\in\mathcal{F}(n+1)$, $f\neq f\cap[0,\max\mathcal{F}(n)]$.
  \end{itemize}
\end{definition}
%The second and third conditions guarantee that, between $\mathcal{F}(n)$ and $\mathcal{F}(n+1)$, each element $f\in\mathcal{F}(n)$ either remains constant in $\mathcal{F}(n+1)$ or is properly extended.  The fourth condition promises that at least one element is properly extended, and the fifth says that once an element stays constant, it is never extended again.   
We will reserve caligraphic $\mathcal{F}$ for unbounded trees (that is, where the domain is unbounded), and write $F$ when the tree is finite---that is, when $\dom(F)=[0,k]$ for some $k$.

\begin{definition}
  When $F$ is a tree of finite sets with domain $[0,k]$, a \emph{branch} of $F$ is an element of $F(k)$.

  When $\mathcal{F}$ is a tree of finite sets with domain $\mathbb{M}$, a \emph{path} of $\mathcal{F}$ is an unbounded set $\Lambda$ such that, for each $n\in\mathbb{M}$, $\Lambda\cap[0,\max\mathcal{F}(n)]\in\mathcal{F}(n)$.
\end{definition}

Our construction will produce an unbounded tree of finite sets which preserves $I\Sigma_1$; in particular, this ensures that $\mathcal{F}$ has an unbounded path.  Our construction will also ensure that every unbounded path of $\mathcal{F}$ preserves $I\Sigma_2$ and is a set on which $c$ is fallow.

In order to manage the syntactic complexity of the definition, we need to introduce our forcing conditions in stages.  The first stage, the pre-pre-condition, captures the computable part of our definition.

\begin{definition}
  A \emph{pre-pre-condition} is a tuple $(F,F_\dagger,c_*,S,u,U,W,V)$ where:
  \begin{enumerate}[label=c.1.\arabic{enumi}]
  \item $F$ and $F_\dagger$ are trees of finite sets with domain $[0,k]$,
  \item for each $n$, $F_\dagger(n)\subseteq F(n)$,
  %\item $F_\dagger$ is growing,
  %\item if $n<k$ and $f\in F(n)$ then there is a $g\in F(n+1)$ with $f\subsetneq g$,
  \item if $f\in F(n)$ then $c$ is fallow on $f$,
  \item $c_*: \bigcup_n\bigcup F(n)\rightarrow \{0,1\}$ is a function such that for any $f\in F(n)$ and $a,b\in f$, $c_*(a)\in\{c_*(b),c(a,b)\}$,
  \item $S$ is a family of partitions of size $r$, $\max F<\min S$,
  \item $u$ is a map from $[0,r)$ to the branches of $F$,
  \item $\emptyset\neq U\subseteq [0,r)$,
  \item $u\upharpoonright U$ is a map to the branches of $F_\dagger$,
  \item $W=(W_0,\ldots,W_w)$ is a sequence such that each $W_i$ is a collection $W_i^x$ of finite sets of the form $\{f\mid \text{for each }x'\leq x\text{ there is some initial segment }g\sqsubseteq f\text{ such that }g\in W^*_{i,x'}\}$ where $W^*_i$ is uniformly computable,
  \item $V$ is a finite collection of pairs $(i,x)$ with $i\leq w$,
  \item if $f$ is a branch of $F_\dagger$ then $f\in\bigcap_{(i,x)\in V}W_i^x$.
  % \item at least one branch $f$ of $F$ is contained in $\bigcap_{(i,x)\in V}W_i^x$,
  % \item $\theta:V\rightarrow[0,k]$ is injective,
  % \item if $f\cap[0,\max F(\theta((i,x)))]$ and $f\in W_i^x$ then $f\cap[0,\max F(\theta((i,x)))]\in W_i^x$.
%  \item if $(i,x)\in V$, $f\in F(n)\cap W_i^x$, and $n>\theta((i,x))$ then $f\cap[0,\max F(\theta((i,x)))]\in W_i^x$,
  %\item there is a branch $f$ from $F(k)$ contained in $\bigcap_{(i,x)\in V}W_i^x$,
%  \item $u:[0,r]\rightarrow F(k)$ and $\rng(u)=F(k)\cap \bigcap_{(i,x)\in V}W_i^x$,
  %\item if $f\in F(\theta((i,x)))$ and, for all $(i',x')$ with $\theta((i',x'))\leq\theta(i,x)$, $f\in W_{i'}^{x'}$ then $f\not\in F(\theta((i,x))-1)$.
  \end{enumerate}

%We call an element of $F(k)$ a \emph{branch} of $F$.  
%We call a branch $f$ \emph{attested} if $f\in \bigcap_{(i,x)\in V}W_i^x$.
\end{definition}
Being a pre-pre-condition is a computable property.  We write $W\upharpoonright i$ for the list $(W_0,\ldots,W_{i-1})$ and $V\upharpoonright i=\{(j,x)\in V\mid j<i\}$.

This definition is rather complicated, and needs some explanation. 

The first three pieces, $F,F_\dagger,c_*$, describe the finite part of our condition---the finite initial segment of our eventual generic.  $F_\dagger$ is the actual tree we are constructing: our goal is to simultaneously construct $F_\dagger$ so that it preserves $I\Sigma_1$ and so that each branch preserves $I\Sigma_2$ and gives subset on which our coloring is fallow.  The tree $F$ is wider than $F_\dagger$---it may have additional branches---and exists for technical bookkeeping reasons.  $c_*$ is a guess at the limit coloring $c_{\infty}$.

The next three pieces, $S,u,U$, describe a Matthias restraint.  $S$ is the restraint: branches of $F$ (and, in particular, of $F_\dagger$) should only be extended within a single component of $S$.  The function $u$ tells us which component of $S$ can be used to extend a branch of $F$; this is defined so that each branch $f\in F(k)$ has a set $u^{-1}(f)$ of components in which it is allowed to extend.  In an extension of this pre-pre-condition, we will require that all extensions of $f$ be contained some $g_j$ with $j\in u^{-1}(f)$.  $U\subseteq[0,r)$ is the set of ``live'' extensions---the ones which will actually be part of $F_\dagger$.

The final two pieces, $W,V$, describe our progress towards matching various requirements.  We need these as part of our conditions because we will have to use them to help track which branches of certain partitions are live.  (This is one of the main new complications in the construction.)   Roughly speaking, each $W_i^x$ represents some condition we are putting on our eventual path; eventually we should have $\Lambda\in\bigcap_{i,x}W_i^x$.  In practice, $W^*_{i,x}$ will have the form $\{h\mid \exists z\leq |h|,h'\subseteq h\phi(x,z,h')\}$, so $f\in W_i^x$ will mean that, for all $x'\leq x$, there is an initial segment $f'\subseteq f$ and a $z\leq|f|$ so that $\phi(x,z,f')$ holds for some quantifier-free formula $\phi$.

$V$ is the set of pairs $(i,x)$ of conditions which we have already succeeded in enforcing; thus $f\in\bigcap_{(i,x)\in V}W^x_i$.

\begin{definition}
  We say $(F',F'_\dagger,c'_*,S',u',U',W',V')\preceq_\rho (F,F_\dagger,c_*,S,u,U,W,V)$ if:
  \begin{itemize}
  \item $\dom(F)\subseteq\dom(F')$,
  \item for all $n\in\dom(F)$, $F'(n)=F(n)$ and $F'_\dagger(n)=F_\dagger(n)$,
  \item $c_*\subseteq c'_*$,
  \item $S'$ refines $S$ via $\rho$,
  \item for $j'<r'$, $u'(j')$ extends $u(\rho(j'))$,
  \item for every $\{g_{j'}\}_{j'<r'}\in S'(n)$, there is an $\{h_j\}_{j<r}\in S'(n+\min S'-\min S)$ so that for each $j'<r'$, $(u'(j')\setminus u(\rho(j')))\cup g_{j'}\subseteq h_j$,
  \item if $j'\in U'$ then $\rho(j')\in U$,
  \item $W\sqsubseteq W'$,
  \item $V\subseteq V'$.
  \end{itemize}

  We say $(F',F'_\dagger,c'_*,S',u',U',W',V')\preceq (F,F_\dagger,c_*,S,u,U,W,V)$ if there is some $\rho$ so that $(F',F'_\dagger,c'_*,S',u',U',W',V')\preceq_\rho (F,F_\dagger,c_*,S,u,U,W,V)$.
\end{definition}

\begin{definition}
  We say a pre-pre-condition $(F,F_\dagger,c_*,S,u,U,W,V,\theta)$ is a \emph{pre-condition} if:
  \begin{enumerate}[label=c.2.\arabic{enumi}]
%  \item for $a\in \bigcup_n\bigcup F(n)$, $c_*(a)=c_\infty(a)$,
  \item for $a\in \bigcup_n\bigcup F(n)$ and $b\geq\min S$, $c(a,b)=c_*(a)=c_{\infty}(a)$,
%  \item $S(n)$ is non-empty for all $n$,
  \item $S$ is infinite,
  \item for any $j<r$ and $(i,x)\in V$, if $u(j)\not\in W_i^x$ then whenever $\{g_j\}_{j<r}\in S(n)$ and $g'\subseteq g_j$, $u(j)\cup g_j\not\in W_i^x$.
%  \item if $f$ is a branch in $F$, $(i,x)\in V$, and $f\not\in W_i^x$ then whenever $g\in \overline{S}_{i, \{f_j\}_{j\in R}}(n)$, and $g'\subseteq g$, $f\cup g'\not\in W_i^x$.
  \end{enumerate}
\end{definition}
Note that being a pre-condition is a $\Pi_1$ property.  The first property says that our guess $c_*$ is correct; combined with the definition of a pre-pre-condition, this implies that any $f\in F(n)$ is a candidate to be part of a set on which $c$ is fallow.

The third property says that if a branch can be extended to belong to $W_i^x$, it already belongs to $W_i^x$.

\begin{definition}
  We say a pre-condition $(F,F_\dagger,c_*,S, u,U,W,V)$ is a \emph{condition} if:
  \begin{enumerate}[label=c.3.\arabic{enumi}]
  \item $S$ is extensive,
  \item for any $i\leq w,x$, whenever $(F',F'_\dagger,c'_*,S',u',U',W\upharpoonright i,V')\preceq_\rho (F,F_\dagger,c_*,S, u,U,W\upharpoonright i,V\upharpoonright i)$ is a pre-condition, there is a level $n$ so that for every $\{g_j\}_{j<r'}\in S'(n)$ and every partition $g_j=\bigcup_{d'\leq d}g_j^{d'}$, there is a $j<r'$ so that $\rho(j)\in U$, $u'(j)$ is contained in $\bigcap_{(i',x')\in V'}W_{i'}^{x'}$, and there is a $d'\leq d$ and $g'\subseteq g_j^{d'}$ so that $u'(j)\cup g'\in W_i^x$ and $c$ is fallow on $u'(j)\cup g'$.\label{cond_cond_extension}
  \end{enumerate}
\end{definition}
Being a condition is a $\Pi_2$ property.  When considering the pre-condition $(F',F'_\dagger,c'_*,S',u',U',W\upharpoonright i,V')$, the only terms that matter are $F'$, $S'$, $u'$, and $V'$---the others can either be inferred (for example, $c'_*$ must be a suitable restriction of $c_\infty$ and so on), or not not matter (there may be multiple choices for $F'_\dagger$ and $U'$, but they do not affect whether the statement holds)---so we often write $(F',S',u',V')\preceq_\rho (F,F_\dagger,c_*,S, u,U,W\upharpoonright i,V\upharpoonright i)$.

The significance of the second property is that says that we can satisfy $W_i^x$ ``densely'': given any extension $(F',S',u',V')$, we can find suitable branches which extend so they belong to $W_i^x$.

\subsection{Tree Generics}

\begin{definition}
A sequence of conditions (indexed by $M$)
\[(F_0,F_\dagger^0,c_{*}^0,S_0,u_0,U_0,W_0,V_0,\theta_0)\succeq (F_1,F_\dagger^1,c_*^1,S_1,u_1,U_1,W_1,V_1,\theta_1)\succeq\cdots\]
is \emph{generic} if:
\begin{itemize}
\item the sequence is coded (i.e. each bounded initial segment is encoded by an element of $M$),
\item $\bigcup F_i$ has domain $\mathbb{M}$,
\item $\bigcup V_i=\mathbb{M}\times \mathbb{M}$, and
\item for every $i\in\mathbb{M}$ there is a $j\geq i$ with $|U_j|=1$.
\end{itemize}
We say the sequence \emph{begins with} $(F_0,F_\dagger^0,c_*^0,S_0,u_0,U_0,W_0,V_0)$.

%A \emph{tree generic} above a condition $(F,F_\dagger,c_*,S, u,U,W,V) $ is a tree $\mathcal{F}=\bigcup F_\dagger^i$ from some coded sequence beginning with $(F,F_\dagger,c_*,S, u,U,W,V,\theta)$.  %A \emph{branch generic} above $(F,F_\dagger,c_*,S, u,U,W,V)$ is an unbounded path from a tree generic above $(F,F_\dagger,c_*,S, u,U,W,V)$.

%We say $(F,F_\dagger,c_*,S, u,U,W,V)\Vdash_{\mathcal{T}}\phi(G)$ if whenever $\mathcal{F}$ is a tree generic above $(F,F_\dagger,c_*,S, u,U,W,V)$, $\phi(\mathcal{F})$ holds.  %We say $(F,F_\dagger,c_*,S, u,U,W,V)\Vdash_{\mathcal{B}}\phi(G)$ if whenever $\Lambda$ is a branch generic above $(F,c_*,S, u,U,W,V)$, $\phi(\Lambda)$ holds.
\end{definition}

We must show that there exist generic sequences at all; specifically, that we can always extend a condition to extend the domain of $F$, and that we can always add elements to $V$.  (Technically, we must also ensure that the $W_i$ represent longer and longer sequences, so that for each $m\in\mathbb{M}$ there is an $i$ so that $W_i$ is a sequence of length $\geq m$, but if all we want to do is construct a generic, this could be accomplished by padding the $W_i$ with trivial sets.  In the next subsection we will show that we can construct generics where the $W_i$ are chosen in a more useful manner.)

\begin{lemma}
  Let $(F,F_\dagger,c_*,S, u,U,W,V)$ be a condition and let $i\leq |W|$.  Then for any $x$, there is a condition $(F',F'_\dagger,c'_*,S', u',U',W,V\cup\{(i,x)\})\preceq (F,F_\dagger,c_*,S, u,U,W,V)$ with the domain of $F'$ strictly larger than the domain of $F$.
\end{lemma}
\begin{proof}
We wish to consider those sets $L\subseteq[0,r-1]$ such that there exists a $\{g_j\}_{j<r}\in S(m)$ such that:
  \begin{itemize}
  \item $\{g_j\}_{j<r}$ is permanent,
  \item for each $j\in L$ and every partition $g_j=\bigcup_{d'\leq d} g_j^{d'}$, there is a $d'$ and a non-empty $g'\subseteq g_j^{d'}$ so that $u(j)\cup g'\in W^x_i$ and $c$ is fallow on $u(j)\cup g'$.
  \end{itemize}
This is a $\Sigma_2$ property, so we may choose some such $\{g_j\}_{j<r}\in S(m)$ maximizing the size of $L$, and we may make this choice with $m$ sufficiently large.  %Note that if $u(j)\in W^x_i$ then, since $W^x_i$ is closed under initial segments, we certainly have $u(j)\in L$.
By restricting $S$, we may assume that whenever $\{h_j\}_{j<r}\in S(n)$, $\{g_j\}_{j<r}\subseteq\{h_j\}_{j<r}$.

For each $j$, write $f_j=u(j)$.  For $j\in L$, $c_{\infty}$ induces a partition $g_j=\bigcup_{d'\leq d}g_j^{d'}$, so we may choose some $d'_j\leq d$ and $g'_j\subseteq g_j^{d'_j}$ with $f_j\cup g'_j\in W^x_i$ and $c$ fallow on $f_j\cup g'_j$ by \ref{cond_cond_extension}.

We wish to obtain $S_1$ surjectively refining $S$ so that whenever $\{h_{j'}\}_{j'<r_1}\in S_1(n)$, for each $j'$ with $\rho(j')\not\in L$ and $h'\subseteq h_{j'}$, if $c$ is fallow on $f_j\cup h'$ then $f_j\cup h'\not\in W_i^x$.  

We do this by repeated application of Lemma \ref{thm:refine_partitions}, once for each $j\not\in L$: for any $j\not\in L$, let $S^*_j(n)$ consist of those partitions $\{h_{j^*}\}_{j^*<d+1}$ such that for every $j^*> 0$, there is no $h'\subseteq h_{j^*}$ such that $f_j\cup h'\in W_i^x$ and $c$ is fallow on $f_j\cup h'$.  Since $j\not\in L$, $S^*$ is infinite: for any $m'$, choose a permanent $\{g'_j\}_{j<r}\in S(m')$ with $\{g_j\}_{j<r}\subseteq\{g'_j\}_{j<r}$.  Then for each $j<r$ with $j\not\in L$, $g'_j$ must have a partition $g'_j=\bigcup_{j^*\leq d}g'_{j^*}$, and therefore the partition $h'_0=\bigcup_{j'\neq j}g'_{j'}$ and $h'_{j^*+1}=g'_{j^*}$ belongs to $S^*_j(n)$.

Then one application of Lemma \ref{thm:refine_partitions} gives us the property we need for $j'$ with $\rho(j')=j$; repeating this for each $j\not\in L$ gives the desired $S_1$.  (The defining property of $S_1$ is $\Pi_1$, so we can carry out this iteration.)

% Define $S_0$ to be the refinement of $S$ of size $|L|+d(r-|L|)$ with $\min S_0=\min S$ obtained by choosing some $\rho:[0,|L|+d(r-|L|)-1]\rightarrow[0,r-1]$ so that for $j\in L$, $|\rho^{-1}(j)|=1$, and for $j\not\in L$, $|\rho^{-1}(j)|=d$,  and placing $\{h_j\}_{j<|L|+d(r-|L|)}\in S_1(m)$ if there is an $\{h'_j\}_{j<r}\in S(m)$ such that:
% \begin{itemize}
% \item for $j\in L$ then $h_{\rho^{-1}(j)}=h'_j$,
% \item for $j\not\in L$, $h'_j=\bigcup_{j'\in \rho^{-1}(j)}h_{j'}$ and there is no $g'\subseteq h_{j'}$ such that $f_j\cup g'\in W^x_i$ and $c$ is transitive on $f\cup g'$.
% \end{itemize}
% The maximality of $L$ ensures that $S_0$ is infinite (but not necessarily extensive), and by Lemma \ref{thm:partition_extensive}, we may find an extensive $S_1$ refining $S_0$ via $\rho_1$, and therefore $S_1$ refines $S$ by $\rho'=\rho\circ\rho_1$.

Choose some permanent $\{h_{j'}\}_{j'<r'}\in S_1(m')$ with $m'$ sufficiently large that, for all $j\not\in L$ and every $j'\in\rho^{-1}(j)$, $|h_{j'}|\geq 1$.  We can now define $F'$ to extend $F$ by one additional level containing:
\begin{itemize}

\item for each $j\in L$, the branch $f_j\cup g'_j$,
\item for each $j\not\in L$, the branches $f_j\cup \{\min h_{j'}\}$ for every $j'\in\rho^{-1}(j)$.
\end{itemize}
We can further refine $S_1$ to $S'$ of the same size by restricting to extensions of $\{h_{j'}\}_{j'<r'}$ and truncating so that $\min S'$ is large enough that every element of $F'$ has achieved its limit color; by abuse of notation, we say $S'$ refines $S$ by $\rho$ as well.  We can then define $u'$ by:
\begin{itemize}
\item if $\rho(j')\not\in L$ then $u'(j')=f_{\rho(j')}\cup\{\min h_{j'}\}$,
\item if $\rho(j')\in L$ then $u'(j')=f_{\rho(j')}\cup g_{\rho(j')}$.
\end{itemize}
We take $c'_*(i,\cdot)$ to be the restriction of $c_{i,\infty}$ to elements of $F'$.

Finally, we let $F'_\dagger$ extend $F_\dagger$ by setting $F'_\dagger(k+1)$ to consist of those branches $u'(j')$ of $F'(k+1)$ which are in $W_i^x$ and such that $\rho(j')\in U$.  We take $U'$ to be exactly those $j'$ such that $u'(j')$ is a branch of $F'_\dagger(k+1)$.

We claim that $(F',F'_\dagger,c'_*,S',u',U',W,V\cup\{(i,x)\})$ is a condition.  Most properties are clear from the discussion above.  The fact that $F'_\dagger$ has a branch---that is, that $L\cap U\neq\emptyset$---follows immediately from the fact that $(F,F_\dagger,c_*,S,u,U,W,V)$ was a condition.

We must verify the extensions property of conditions, \ref{cond_cond_extension}.  Let $i^*\leq W$ and $x^*$ be given, and suppose $(F'',S'',u'',V'')\preceq_{\rho'}(F',F'_\dagger,c'_*,S',u',U',W\upharpoonright i^*,(V\cup\{(i,x)\})\upharpoonright i^*)$.  If $i^*\leq i$ then we also have $(F'',S'',u'',V'')\preceq (F,F_\dagger,c_*,S, u,U,W,V)$ and the necessary property follows because $(F,F_\dagger,c_*,S, u,U,W,V)$ is a condition.

So we assume $i<i^*$.  Because $(F,F_\dagger,c_*,S, u,U,W,V)$ is a condition, there is a level $n$ so that for every $\{g_j\}_{j<r''}\in S''(n)$ and every partition $g_j=\bigcup_{d'\leq d}g_j^{d'}$, there is a $j<r'$ so that $\rho(\rho'(j))\in U$ and $u''(j)$ is contained in $\bigcap_{(i',x')\in V''} W_{i'}^{x'}$, and has the other necessary properties.  By the construction of $F'$ and $S'$, since $u''(j)$ is contained in $W_i^x$, so is $u'(\rho'(j))$ (that is, the restriction of $u''(j)$ to a branch through $F'$), and therefore $\rho'(j)\in U'$ as needed.
\end{proof}

In particular, using $\Pi_2$ induction, for any $V'\supseteq V$ with $i\leq w$ for all $(i,x)\in V'$, we can find a pre-condition $(F',c'_*,S',u',U',W,V')\preceq(F,c_*,S,u,U,W,V)$.  There are two additional properties we need for this to be a condition; the second is already guaranteed if $(F,c_*,S,u,U,W,V)$ is a condition (because we are not extending $W$), and there is always an extensive $S''$ refining $S'$, so we can find a condition $(F',c'_*,S',u'',U'',W,V')\preceq(F,c_*,S,u,U,W,V)$.

Next we show that we can, unboundedly often, arrange to have $|U|=1$.  The crucial idea is that we can actually tell, at intermediate stages of our construction, which branches are necessary.  We do not see how to do this uniformly enough to avoid intermediate stages which allow $|U|>1$ (at least, not without substantial bookkeeping complications), but we do not need to wait for the construction to finish to identify these branches.

\begin{lemma}
  Let $(F,F_\dagger,c_*,S,u,U,W,V)$ be a condition and let $\phi(y,z,G)$ be a quantifier-free formula.  Then there is a condition $(F',F'_\dagger,c'_*,S',u',U',W',V',)\preceq(F,F_\dagger,c_*,S,u,U,W,V)$ such that $|U'|=1$.
\end{lemma}
The argument is rather technical because conditions have many pieces, which makes them difficult to adjust, but the underlying argument is not so complicated.  We shrink $U$ to make it as small as possible while remaining a condition.  If the result is that $|U|=1$, we are finished.  Otherwise, we pick a $j_0\in U$ and remove it; since $U\setminus\{j_0\}$ is not a condition, there must be some extension and an $i,x$ witnessing the failure to be a condition.  This should mean that $j_0$ is necessary---that there are extensions and a choice of $(i,x)$ which could force us into the branch represented by $j_0$.  But this suggests that we should get a condition when we restrict to just $U=\{j_0\}$---otherwise there should also be an extension and a choice of $(i',x')$ which forces us off the branch represented by $j_0$.  But, if we arrange things appropriately, this will give us a contradiction when we use the previous lemma to try to find extensions in $W_i^x\cap W_{i'}^{x'}$.

%This suggests that we should be able to take $\{g_j\}_{j<r}\in S(n)$ and, for each $j\in U\setminus\{j_0\}$, partition in $d$ pieces so that there are no extensions in $W_i^x$.  (Some tweaking of branches to make this precise causes most of the complications.)  After replacing $S$ with this refinement, we then argue that we can take $U=\{j_0\}$ and obtain a condition: if not, there would be an $i',x'$ witnessing this failure.  But then we could also partition the $j_0$ component into $d$ pieces so that no extension belongs to $i',x'$.  The result is that we can no longer simultaneously find branches in $W_i^x$ and $W_{i'}^{x'}$---such a branch cannot be from one of the components which refined the $j_0$ component (because such branches cannot belong to $W_{i'}^{x'}$), but such a branch cannot be from a component refining any other components (because such branches cnanot belong to $W_i^x$).  This contradicts the fact that we began with a condition.
\begin{proof}
We first take an arbitrary proper extension $(\hat F,\hat F_\dagger,\hat c_*,\hat S,\hat u,\hat U,W,V)\preceq (F,F_\dagger,c_*,S,u,U,W,V)$; specifically, we need $k=|\dom(\hat F)|>|\dom(F)|$, because we will need to modify $\hat F(k)$ in the course of finding our condition.

For any $U_0\subseteq\hat U$, let $\hat F_\dagger^{U_0}$ be the tree with $\hat F_\dagger^{U_0}(n)=\hat F_\dagger(n)$ for $n<k$ and $\hat F_\dagger^{U_0}(k)=\{\hat u(j)\mid j\in U_0\}$.  We may choose $U_0\subseteq \hat U$ minimal so that $(\hat F,\hat F^{U_0}_\dagger,\hat c_*,\hat S,\hat u,U_0,W,V)$ is a condition.  If $|U_0|=1$, we are finished, so suppose not.

Choose any $j_0\in U_0$, so $(\hat F,\hat F^{U_0\setminus\{j_0\}}_\dagger,\hat c_*,\hat S,\hat u,U_0\setminus\{j_0\},W,V)$ is not a condition.  Therefore there must be some pre-condition $(F',S',u',V')\preceq_\rho (\hat F,\hat F^{U_0\setminus\{j_0\}}_\dagger,\hat c_*,\hat S,\hat u,U_0\setminus\{j_0\},W\upharpoonright i,V\upharpoonright i)$ and some $x$ so that, for every $n$, there is a $\{g_j\}_{j<r'}\in S'(n)$ so that, for every $j<r'$ such that $\rho(j)\in U_0\setminus\{j_0\}$ and $u'(j)\in\bigcap_{(i',x')\in V'}W_{i'}^{x'}$, there is a partition witnessing the failure to be a condition.  Furthermore, because $p'$ is a pre-condition, if $u'(j)\not\in W_{i'}^{x'}$ then no extension consistent with $S'$ will be in $W_{i'}^{x'}$.

We define two modifications as follows.  We define $\hat F'$ by $\hat F'(n)=\hat F(n)$ for $n<k$ and taking $\hat F'(k)$ to consist of the branches of $\hat F'$ of the form $u'(\rho(j))$ for $j<\hat r'$ such that $\rho(j)\in U_0\setminus\{j_0\}$, together with all branches of $\hat F$ not extended by such a branch.  We take $\hat c'_*$ to be the restriction of $c_\infty$ to $\hat F'$.

We define $\hat S'$ as follows: for $j\not\in U_0$, the $j$-th component of $\hat S'$ is the $j$-th component of $S$.  For $j\in U_0$, we replace the $j$-th component of $S$ with the corresponding components of $S'$, each of which is then partitioned, as in Lemma \ref{thm:refine_partitions}, into $d$ components, some of which may then be eliminated to make the resulting $\hat S'$ extensive.  $\hat u'$ is the natural composition with $u$.

Take $\hat F'_\dagger$ so that the branches of $\hat F'_\dagger$ are exactly the branches of $\hat F'$ which extend branches of $\hat F_\dagger$, and take $\hat U^+$ to consist of all $j$ which refine $j'\in U_0$.  Then $(\hat F',\hat F'_\dagger,\hat c'_*,\hat S',\hat u',\hat U^+,W,V)$ is a condition---the only thing to check is \ref{cond_cond_extension}, and this follows because if $(F'',S'',u'',V'')\preceq (\hat F',\hat F'_\dagger,\hat c'_*,\hat S',\hat u',\hat U^+,W,V)$ then there is a modification $\tilde F''$ with the same branches so that $(\tilde F'',S'',u'',V'')\preceq (\hat F,\hat F_\dagger,\hat c_*,\hat S,\hat u,U_0,W,V)$, and since $(\hat F,\hat F_\dagger,\hat c_*,\hat S,\hat u,U_0,W,V)$ is a condition, so we can find a suitable extension of a branch of $\tilde F''$, and since only the branches of $\tilde F''$ matter, also a branch of $F''$ has a suitable extension.

Now take $\hat U'=\{j_0\}$.  We claim that $(\hat F',\hat F'_\dagger,\hat c'_*,\hat S',\hat u',\hat U',W,V)$ is also a condition.  Suppose not; then there is a pre-condition $(F'',S'',u'',V'')\preceq_{\rho''} (\hat F',\hat F'_\dagger,\hat c'_*,\hat S',\hat u',\hat U',W\upharpoonright i',V\upharpoonright i')$ and an $x'$ witnessing this failure.

We modify this to become a condition.  By refining $S''$ to be extensive, applying Lemma \ref{thm:refine_partitions} to each $j\in U''$, we obtain $\hat S''$ so that no extension of a $u''(j)$ has a suitable extension belonging to $W_{i'}^{x'}$.  We take $\hat u''$ to be the natural composition with $u''$.  We take $\hat F''_\dagger$ to consist of all branches (with initial segments at suitable levels) belonging to $\bigcap_{(i'',x'')\in V''}W_{i''}^{x''}$ which extend a branch of $\hat F_\dagger^+$, and take $\hat U''$ to consist of all $j$ whose image is in $\hat U^+$.

We claim that $(F'',\hat F''_\dagger,c''_*,\hat S'',\hat u'',\hat U'',W,V'')$ is a condition.  It is easily seen to be a pre-pre-condition by the definition of the components.  Since $(F'',S'',u'',V'')$ was a pre-condition, so is $(F'',\hat F''_\dagger,c''_*,\hat S'',\hat u'',\hat U'',W,V'')$.  We have ensured $\hat S''$ is extensive.  Finally, since $(F'',\hat F''_\dagger,c''_*,\hat S'',\hat u'',\hat U'',W,V'')\preceq (\hat F',\hat F^+_\dagger,\hat c'_*,\hat S',\hat u',\hat U^+,W,V)$ and the latter is a condition, so the final property of being a condition holds for $(F'',\hat F''_\dagger,c''_*,\hat S'',\hat u'',\hat U'',W,V'')$.

But by the previous lemma, we must have some condition $(F^*,F^*_\dagger,c^*_*,S^*,u^*,U^*,W,V''\cup V'\cup\{(i,x),(i',x')\})\preceq_{\rho^*} (F'',\hat F''_\dagger,c''_*,\hat S'',\hat u'',\hat U'',W,V'')$.  There must be some $j\in U^*$ so that $u^*(j)\in W_i^x\cap W_{i'}^{x'}\cap\bigcap_{(i'',x'')\in V''}W_{i''}^{x''}$.  Consider the image of $j$ in $\hat U^+$.  We cannot have the image be $j_0$, because we refined so that no extension could belong to $W_{i'}^{x'}$.  So the image of $j_0$ must be in $\hat U^+\setminus\{j_0\}$.  Since the branch belongs to $\bigcup_{(i'',x'')\in V''}W_{i''}^{x''}$, it must be a branch corresponding to a component which we refined so that no extension could belong to $W_i^x$.  This is a contradiction, so we conclude that $(\hat F',\hat F'_\dagger,\hat c'_*,\hat S',\hat u',\hat U',W\upharpoonright i',V\upharpoonright i')$ must have been a condition.
\end{proof}

\subsection{Branch Generics}

\begin{definition}
A \emph{branch generic} is above $(F,F_\dagger,c_*,S, u,U,W,V)$ is the unique unbounded path of $F_\dagger$ from a generic sequence above $(F,F_\dagger,c_*,S, u,U,W,V)$ % such that the sequence
% \[(F_0,F_\dagger^0,c_{*}^0,S_0,u_0,U_0,W_0,V_0,\theta_0)\succeq (F_1,F_\dagger^1,c_*^1,S_1,u_1,U_1,W_1,V_1,\theta_1)\succeq\cdots\]
% for the tree generic has $|U_i|=1$ infinitely often.

 We say $(F,F_\dagger,c_*,S, u,U,W,V)\Vdash\phi(G)$ if whenever $\Lambda$ is a branch generic above $(F,c_*,S, u,U,W,V)$, $\phi(\Lambda)$ holds.
\end{definition}

We need to show that, given $(F,F_\dagger,c_*,S, u,U,W,V)$ and a $\Pi_2$ setence $\forall x\exists y\phi(x,y,G)$, we can extend $(F,F_\dagger,c_*,S, u,U,W,V)$ to force either $\forall x\exists y\phi(x,y,G)$ or its negation, and, crucially, whether we force $\forall x\exists y\phi(x,y,G)$ is determined by a $\Pi_2$ property of $(F,F_\dagger,c_*,S, u,U,W,V)$.

The key property is:
\begin{quote}
  $(\ast)$ For every pre-condition $(F',F'_\dagger,c'_*,S',u',U',W,V')\preceq_\rho (F,S, u,n_0,\{f_j\}_{j\in R},W,V)$ and every $x$ there is an $n$ and a $\{g_j\}_{j<r}\in S(n)$ so that for every partition $g_j=\bigcup_{d'\leq d} g_j^{d'}$ there is a $j$ with $\rho(j)\in U$ so that $u'(j)\in\bigcap_{(i',x')\in V'}W_{i'}^{x'}$, a $d'\leq d$, and a $g'\subseteq g_j^{d'}$ so that $\exists y\phi(x,y,u'(j)\cup g')$ and $c$ is fallow on $u'(j)\cup g'$.
\end{quote}
Note that $(\ast)$ is a $\Pi_2$ property (depending on $(F,S, u,n_0,\{f_j\}_{j\in R},W,V)$ and $\phi(x,y,G)$).

\begin{lemma}\label{thm:choice_of_generic_pos}
  Let $(F,F_\dagger,c_*, S, u,U,W,V)$ be a condition and let $\phi(x,y,G)$ be a quantifier-free formula (with no other free variables, but possibly with numeric and set parameters).  Suppose that $(\ast)$ holds.

%Suppose that for every pre-condition $(F',F'_\dagger,c'_*,S',u',U',W,V')\preceq_\rho (F,S, u,n_0,\{f_j\}_{j\in R},W,V)$ and every $x$ there is an $n$ and a $\{g_j\}_{j<r}\in S(n)$ so that for every partition $g_j=\bigcup_{d'\leq d} g_j^{d'}$ there is a $j$ with $\rho(j)\in U$ so that $u'(j)\in\bigcap_{(i',x')\in V'}W_{i'}^{x'}$, a $d'\leq d$, and a $g'\subseteq g_j^{d'}$ so that $\exists y\phi(x,y,u'(j)\cup g')$ and $c$ is transitive on $u'(j)\cup g'$.

Then there is a $W^*$ so that $(F,F_\dagger,c_*, S, u,U,W^\frown\langle W^*\rangle,V)$ is a condition and $(F,F_\dagger,c_*, S, u,U,W^\frown\langle W^*\rangle,V)\Vdash\forall y\exists z\phi(y,z,G)$.
\end{lemma}
\begin{proof}
  Take $h\in W_*^x$ if there is an $h'\subseteq h$ so $\exists y\leq |h|\phi(x,y,h')$.  Then $(F,F_\dagger,c_*, S, u,U,W^\frown\langle W^*\rangle,V)$ is a condition, essentially by the definition.

  If $\Lambda$ is a branch generic, in particular, for every $x$ there is some initial segment of $\Lambda$ belonging to $W^*_x$, and therefore there is some $y$ so that $\phi(x,y,\Lambda)$, as needed.
\end{proof}

\begin{lemma}\label{thm:choice_of_generic_neg}
   Let $(F,F_\dagger,c_*, S, u,U,W,V)$ be a condition and let $\phi(x,y,G)$ be a quantifier-free formula.  Suppose that $(\ast)$ fails.
%the condition in the previous lemma fails; that is, there is a pre-condition $(F',S',u',V')\preceq_\rho (F,F_\dagger,c_*, S, u,U,W,V)$ and an $x$ so that for every $n$ and $\{g_j\}_{j<r}\in S(n)$ there is a partition $g_j=\bigcup_{d'\leq d}g_j^{d'}$ so that for any $j$ with $\rho(j)\in U$ and $u'(j)\in\bigcap_{(i',x')\in V'}W_{i'}^{x'}$, any $d'\leq d$, and any $g'\subseteq g_j^{d'}$ so $c$ is transitive on $u'(j)\cup g'$, there is no $y$ with $\phi(x,y,u'(j)\cup g')$.

Then there is a condition $(F',F''_\dagger,c'_*,S'', u'',U'',W,V')\preceq (F,F_\dagger,c_*, S, u,U,W,V)$ such that $(F',F''_\dagger,c'_*,S'', u'',U'',W,V')\Vdash\exists x\forall y\neg\phi(x,y,G)$.
\end{lemma}
\begin{proof}
Since $(\ast)$ fails, we may choose a precondition $(F',F'_\dagger,c'_*,S',u',U',W,V')\preceq_\rho (F,F_\dagger,c_*, S, u,U,W,V)$ and an $x$ so that for every $n$, there is a $\{g_j\}_{j<r}\in S(n)$ and a partition $g_j=\bigcup_{d'\leq d}g_j^{d'}$ so that for any $j$, any $d'\leq d$, and $g'\subseteq g_j^{d'}$ such that $\rho(j)\in U$, $u'(j)\in\bigcap_{(i',x')\in V'}W_{i'}^{x'}$, and $c$ is fallow on $u'(j)\cup g'$, there is no $y$ so that $\phi(x,y,u'(j)\cup g')$ holds.

Choose $F''_\dagger$ extending $F_\dagger$ and contained in $F'$ so that the branches of $F''_\dagger$ are exactly those extending branches of $F_\dagger$ and contained in $\bigcap_{(i',x')\in V'}W_{i'}^{x'}$.

We wish to obtain $S''$ surjectively refining $S'$ by $\rho'$ so that whenever $\{h_{j'}\}_{j'<r''}\in S''(n)$, if $u'(\rho'(j'))$ is a branch of $F''_\dagger$ then there is no $h'\subseteq h_{j'}$ and $y$ so that $c$ is fallow on $u'(\rho'(j))\cup h'$ and $\phi(x,y,u'(\rho'(j'))\cup h')$ holds.  

We do this by repeated application of Lemma \ref{thm:refine_partitions}, once for each $j$ with $u'(j)$ a branch of $F''_\dagger$ (so $\rho(j)\in U$ and $u'(j)\in\bigcap_{(i',x')\in V'}W_{i'}^{x'}$).  For such a $j$, let $S^*_j$ be the family so $\{h_{j^*}\}_{j^*<d+2}\in S^*_j(n)$ if there is a $\{g_j\}_{j<r}\in S(n)$ with $h_0=\bigcup_{j'\neq j}g_{j'}$ and so that for each $0<d'+1\leq d+1$ and $g'\subseteq h_{d'+1}$ with $c$ fallow on $u'(j)\cup g'$, there is no $y\leq n$ so that $\phi(x,y,u'(j)\cup g')$.  By assumption, $S^*_j$ is infinite, so Lemma \ref{thm:refine_partitions} gives a refinement with the property we want for $j'\in(\rho')^{-1}(j)$.  By iterating this, we obtain the desired $S''$.  We take $u''(j')=u'(\rho'(j'))$ and $U''$ to be all $j'$ with $u''(j')$ a branch of $F''_\dagger$.

Then $(F',F''_\dagger,c'_*,S'', u'',U'',W,V') \preceq (F,F_\dagger,c_*, S, u,U,W,V)$ and $(F',F''_\dagger,c'_*,S'', u'',U'',W,V')\Vdash\forall y\neg\phi(x,y,G)$, so $(F',F''_\dagger,c'_*,S'', u'',U'',W,V')\Vdash\forall y\neg\phi(x,y,G)\Vdash\exists x\forall y\neg\phi(x,y,G)$.
\end{proof}

\subsection{Putting it Together}

\begin{lemma}
  Let $(\mathbb{M},\mathcal{M})\vDash\mathsf{RCA}_0+I\Sigma^0_2$ with $|\mathbb{M}|=|\mathcal{M}|=\aleph_0$ and let $c$ be a stable coloring.  Then there is a countable $\mathcal{M}'\supseteq\mathcal{M}$ so that $(\mathbb{M},\mathcal{M}')\vDash\mathsf{RCA}_0+I\Sigma^0_2$ and there is an $S\in\mathcal{M}$ which is unbounded and so that $c\upharpoonright[S]^2$ is fallow.
\end{lemma}
\begin{proof}
  Take the condition $(F^0,F^0_\dagger,c^0_*,S^0,u^0,U^0,W^0,V^0)$ where $F^0=F^0_\dagger$ is the tree with domain $[0,-1]=\emptyset$, $c^0_*$ is the empty function, $S^0(n)=\{[0,n]\}$ for all $n$ is a family of size $1$, $u^0(0)$ is the unique (empty) branch of $F^0$, $U^0=\{0\}$, $W^0$ is the empty sequence, and $V^0=\emptyset$.  This is trivially a condition.

  We construct a generic sequence of conditions $(F^0,F^0_\dagger,c^0_*,S^0,u^0,U^0,W^0,V^0)\preceq (F^1,F^1_\dagger,c^1_*,S^1,u^1,U^1,W^1,V^1)\preceq\cdots$ using Lemmata \ref{thm:choice_of_generic_pos} and \ref{thm:choice_of_generic_neg} (and $I\Sigma^0_2$) to ensure that, for every $\Pi^0_2$ formula $\forall y\exists z\phi(y,z,G)$, there is an $n$ so that either $(F^n,F^n_\dagger,c^n_*,S^n,u^n,U^n,W^n,V^n)\vDash\forall y\exists z\phi(y,z,G)$ or $(F^n,F^n_\dagger,c^n_*,S^n,u^n,U^n,W^n,V^n)\vDash\exists y\forall z\neg\phi(y,z,G)$.

Let $S$ be the branch generic corresponding to this sequence.  Then we may take $\mathcal{M}'\supseteq\mathcal{M}\cup\{S\}$ by closing under computable definitions, so $(\mathbb{M},\mathcal{M}') \vDash\mathsf{RCA}_0+I\Sigma^0_2$, and $c\upharpoonright[S]^2$ is fallow.
\end{proof}

\begin{theorem}
  $\mathsf{fSEM}_{<\infty}$ is $\Pi^1_1$-conservative over $\mathsf{RCA}_0+I\Sigma^0_2$.
\end{theorem}
\begin{proof}
  Let $(\mathbb{M},\mathcal{M})\vDash\mathsf{RCA}_0+I\Sigma^0_2$ with $|\mathbb{M}|=|\mathcal{M}|=\aleph_0$.  By repeated applications of the preceeding lemma to construct $\mathcal{M}\subseteq\mathcal{M}_1\subseteq\mathcal{M}_2\subseteq\cdots$, we can take $\mathcal{M}_\omega=\bigcup_n\mathcal{M}_n$ so that $(\mathbb{M},\mathcal{M}_\omega)\vDash\mathsf{RCA}_0+I\Sigma^0_2$ and for every coloring $c\in\mathcal{M}_\omega$, there is some unbounded set $S\in\mathcal{M}_\omega$ so that $c\upharpoonright[S]^2$ is fallow.  Therefore $(\mathbb{M},\mathcal{M}_\omega)\vDash\mathsf{fSEM}_{<\infty}$.
\end{proof}

\section{Finitary Conservation}

In this section we give a finitary conservation proof, showing that $\WKLo+\EMinf$ is $\tilde\Pi^{0}_{3}$-conservative over $\RCAo$.  (Recall that a formula is $\tilde\Pi^0_3$ if it has the form $\forall X\phi(X)$ where $\phi$ is a $\Pi^0_3$ formula.)  The idea is that we approximate $\EMinf$ by a finite combinatorial principle which is a ``density notion'' for $\EMinf$.  We will follow the argument developed in \cite{PY2018, KY-unpublished}.

\subsection{$\alpha$-largeness and $\EMinf$-density}

We fix a primitive recursive notation for ordinals below $\ome^{\ome}$ as follows.
We consider ordinals described by the Cantor normal form $\alpha=\sum_{i<k}\ome^{n_{i}}$ 
where $n_{i}\in\N$ and $n_{0}\ge\dots\ge n_{k-1}$.
We write $1$ for $\ome^{0}$, and $\ome^{n}\cdot k$ for $\sum_{i<k}\ome^{n}$.
For a given $\alpha<\ome^{\ome}$ and $m\in \N$, define $0[m]=0$, $\alpha[m]=\beta$ 
if $\alpha=\beta+1$ and $\alpha[m]=\beta+\ome^{n-1}\cdot m$ if $\alpha=\beta+\ome^{n}$ for some $n\ge 1$.

\begin{definition}[$\II$, largeness]
Let $\alpha<\ome^{\ome}$, and let $n,k,m\in\N$.
\begin{enumerate}
 \item A set $X = \{ x_0 < \dots < x_{\ell-1} \}\subseteq_{\fin}\N$ 
is said to be \emph{$\alpha$-large} if $\alpha[x_{0}]\dots[x_{\ell-1}]=0$. 
In other words, any finite set is $0$-large, and $X$ is said to be $\alpha$-large if
\begin{itemize}
 \item $X\setminus \{\min X\}$ is $\beta$-large if $\alpha=\beta+1$,
 \item $X\setminus \{\min X\}$ is $(\beta+\ome^{n-1}\cdot\min X)$-large if $\alpha=\beta+\ome^{n}$.
\end{itemize}
 \item A set $X \subseteq_{\fin}\N$ 
is said to be \emph{$\EMinf$-$\alpha$-large} if for any $P:[X]^{n}\to \min X$, there exists $Y\subseteq X$ such that $P$ is fallow on $[Y]^{2}$ and $Y$ is $\alpha$-large.
\end{enumerate}
\end{definition}
The above definition of $\omega^{n}$-largeness causes a minor trouble if $\min X=0$.
To avoid this and simplify the notation, we will always consider a finite set $X\subseteq_{\fin}\N$ with $\min X>3$.

\begin{definition}[$\II$, $\EMinf$-density]\label{def-density-with-indicator}
We define the notion of \textit{$\EMinf$-$m$-density} for a finite set~$X\subseteq\N$ inductively as follows.
First, a set~$X$ is \textit{$\EMinf$-$0$-dense} if it is $\ome$-large and $\min X>3$.
Assuming the notion of~$\EMinf$-$m$-density is defined, 
a set~$X$ is \textit{$\EMinf$-$(m+1)$-dense} if
\begin{itemize}
 \item for any~$P:[X]^{2}\to \min X$,
there is an $\EMinf$-$m$-dense set~$Y \subseteq X$
such that $P$ is fallow on $[Y]^{2}$, and,
 \item for any partition $Z_{0}\sqcup\dots \sqcup Z_{\ell-1}=X$ such that $\ell\le Z_{0}<\dots<Z_{\ell-1}$, one of the $Z_{i}$'s is $\EMinf$-$m$-dense.
\end{itemize}
Note that there exists a $\Delta^{0}_{0}$-formula $\theta(m,X)$ saying that ``$X$ is $\EMinf$-$m$-dense.''
\end{definition}
Be aware that a density notion for $\EM$ is also defined in \cite{MR3659408}, but our density notion is different from theirs.
Indeed, the second condition above requires that $\EMinf$-$m$-dense set is at least $\ome^{m+1}$-large, thus $\EMinf$-$(m+1)$-dense set needs to be at least $\EMinf$-$\ome^{m+1}$-large.
Here, the second condition is needed to make it compatible with the indicator argument (see \cite[Section 3]{PY2018}).
Now, what we need for the conservation result in this section is the following.
\begin{theorem}
\label{thm:density-main-weak-version}
For any $m\in\omega$, there exists $m\in\omega$ such that $\II$ proves that
if $X\subseteq_{\fin}\N$ is $\ome^{n}$-large, then $X$ is $\EMinf$-$m$-dense.
\end{theorem}
We see this in Subsection~\ref{subsection:EM-largeness}.

\subsection{Conservation proof}\label{section:pi3-conservation}
To obtain the $\tilde\Pi^{0}_{3}$ conservation theorem for $\EMinf$, we will follow the indicator argument developed in \cite[Section~3]{PY2018}.

\begin{lemma}\label{lem:EM-indicator-construction}
Let $M$ be a countable nonstandard model of $\Ii$, and $X\subseteq M$ is an $M$-finite set which is $\EMinf$-$m$-dense for some $m\in M\setminus\omega$.
Then, there exists a cut $I\subseteq M$ such that $(I,\Cod(M/I))\models\WKLo+\EMinf$ and $X\cap I$ is unbounded in $I$.
\end{lemma}
\begin{proof}
We follow the proof of \cite[Lemma~3.2]{PY2018}.
Let $M\models\Ii$ be a countable nonstandard model, and $X\subseteq M$ be $M$-finite set which is $\EMinf$-$m$-dense for some $m\in M\setminus \omega$.
Let $\{E_{i}\}_{i\in\omega}$ be an enumeration of all $M$-finite sets 
such that each $M$-finite set appears infinitely many times, and
$\{P_{i}\}_{i\in\omega}$ be an enumeration of all $M$-finite functions from $[[0,\max X]]^{2}$ 
to $c_{i}<\max X$ such that each function appears infinitely many times.

We will construct an $\omega$-length sequence of $M$-finite sets 
$X=X_{0}\supseteq X_{1}\supseteq\dots$ so that for each~$i \in \omega$, $X_{i}$ is $\EMinf$-$(m-i)$-dense, 
$P_{i}$ is fallow on $[X_{3i+1}]^{2}$ if $c_{i}<\min X_{3i}$, and $[\min X_{3i+2}+1,\max X_{3i+2}-1]\cap E_{i}=\emptyset$ if $|E_{i}|<\min X_{3i+1}$, and $\min X_{3i+2}<X_{3i+3}$.
For each $i\in\omega$, we do the following.
At the stage $3i$, 
if $\min X_{3i}>c_{i}$, take $X_{3i+1}\subseteq X_{3i}$ so that $P_{i}$ is fallow on $[X_{3i+1}]^{2}$ by the first condition of $\EMinf$-density,
and otherwise, put $X_{3i+1}=X_{3i}$.
At the stage $3i+1$, 
if $\min X_{3i+1}>|E_{i}|$, take $X_{3i+2}\subseteq X_{3i+1}$ so that $[\min X_{3i+2}+1,\max X_{3i+2}-1]\cap E_{i}=\emptyset$ by the second condition of $\EMinf$-density,
and otherwise, put $X_{3i+2}=X_{3i+1}$.
At the stage $3i+2$, put $X_{3i+3}=X_{3i+2}\setminus \{\min X_{3i+2}\}$.

Now, let $I=\sup\{\min X_{i}\mid i\in\omega\}\subseteq_{e} M$.
By the construction of the stages $3i+1$, $I$ is a semi-regular cut, thus $(I,\Cod(M/I))\models\WKLo$.
By the construction of the stages $3i+2$, $X_{i}\cap I$ is infinite in $I$ for any $i\in\omega$.
To check that $(I,\Cod(M/I))\models\EMinf$, we will see the construction of the stages $3i$.
Let $P:[I]^{2}\to c$ be a function which is a member of $\Cod(M/I)$ and $c\in I$.
Then, there exists some $i\in\omega$ such that $c=c_{i}$, $P=P_{i}\cap I$ and $c_{i}<\min X_{3i}$.
Hence $P$ is fallow on $[X_{3i+1}\cap I]^{2}$, and $X_{3i+1}\cap I\in\Cod(M/I)$ is an infinite set in $I$.
\end{proof}
On the other hand, it is not hard to check that $\RCAo$ proves that any infinite set contains $\ome^{k}$-large subset, for any standard natural number $k\in\omega$.
Assuming Theorem~\ref{thm:density-main-weak-version}, we have the following.
\begin{lemma}\label{lem:ome-large-in-RCA}
Let $k\in\omega$ be a standard natural number.
Then, it is provable within $\RCAo$ that for any infinite set $X_{0}\subseteq\N$, there exists a finite set $X\subseteq X_{0}$ which is $\EMinf$-${k}$-dense.
\end{lemma}
Thus, combining Lemmas~\ref{lem:EM-indicator-construction} and \ref{lem:ome-large-in-RCA}, we have the following.
\begin{theorem}
$\WKLo+\EMinf$ is a $\tilde\Pi^{0}_{3}$-conservative extension of $\II$.
\end{theorem}
\begin{proof}
Same as Theorem~3.3 of \cite{PY2018}.
\end{proof}
\begin{cor}
$\WKLo+\EMinf$ does not imply $\III$.
\end{cor}
\begin{proof}
$\III$ is not a $\tilde\Pi^{0}_{3}$-conservative extension of $\II$ since it implies the consistency of $\II$.
\end{proof}

\subsection{Calculation for $\EMinf$-$\ome^{n}$-largeness}\label{subsection:EM-largeness}

In this subsection, we will prove Theorem~\ref{thm:density-main-weak-version}.
We will essentially follow the combinatorial argument in \cite{KY-unpublished}.
In \cite{KS81}, Ketonen and Solovay analyze the Paris-Harrington principle by $\alpha$-largeness notion, and clarify the relation between Paris-Harrington principle and hierarchy of fast growing functions.
The case for $\EMinf$ is the following.
\begin{theorem}[$\II$]\label{thm:EMinf-omega-large}
If $X\subseteq_{\fin}\N$ is $\ome^{3}$-large and $\min X>3$, then it is $\EMinf$-$\ome$-large.
%This is provable within $\II$.
\end{theorem}
\begin{proof}
We follow the idea of \cite[Theorem~10]{MR3659408}.
Let $X$ be $\ome^{3}$-large.
Put $a=\min X$.
Then, $|X|>(a+1)^{a+1}$.
For a given coloring $P:[X]^{2}\to \min X$, we will construct $X_{0}\supseteq X_{1}\supseteq\dots\supseteq X_{a}$ and $a_{0},a_{1},\dots,a_{a}\in X$ as follows.
Put $a_{0}=a$ and $X_{0}=X$.
For a given $X_{i}$, put $a_{i}=\min X_{i}$, and choose $X_{i+1}\subseteq X_{i}\setminus\{a_{i}\}$ to be one of $\{b\in X_{i}\setminus \{a_{i}\}: P(a_{i},b)=c\}$ $(c=0,\dots,a-1)$ so that $|X_{i+1}|>(a+1)^{a-i}$.
Then, for any $0\le i<j<k\le a$, $P(a_{i},a_{j})=P(a_{i},a_{k})$.
Thus, $Y=\{a_{0},a_{1},\dots,a_{a}\}$ is $\ome$-large and $P$ is fallow on $[Y]^{2}$.
\end{proof}
We will generalize the above theorem.
Indeed, we need a version with larger solutions.
Our target theorem is the following, which trivially implies Theorem~\ref{thm:density-main-weak-version}.
\begin{theorem}[$\II$]\label{thm:EMinf-largeness-main}
\begin{enumerate}
 \item If $X\subseteq_{\fin}\N$ is $\ome^{18n}$-large and $\min X>3$, then it is $\EMinf$-$\omega^{n}$-large.
 \item If $X\subseteq_{\fin}\N$ is $\ome^{18^{n}}$-large and $\min X>3$, then it is $\EMinf$-$n$-dense.
\end{enumerate}
\end{theorem}

We will show this theorem by decomposing $\ome^{n}$-large sets.
We first prepare basic lemmas for $\alpha$-large sets.
We will use the following lemmas from \cite{KY-unpublished}.
%We first consider a ``sparse enough'' finite sets.
A set $X$ is said to be \emph{$\alpha$-sparse} if $\min X>3$ and for any $x,y\in X$, $x<y$ implies the interval $[x,y)$ is $\alpha$-large.
One can easily check that if a set $X$ is $\ome^{3}$-sparse then it is quadratic exponentially sparse in the following sense: for any $x,y\in X$, $x<y$ implies $4^{x^{2}}<y$.
Trivially, any subset of an $\alpha$-sparse set is $\alpha$-sparse.
\begin{lemma}[Lemmas~1.3, 2.2 and 2.2 of \cite{KY-unpublished}, $\II$]\label{lem:from-KY}
%\label{lem:alpha-sparse}
\begin{enumerate}
 \item Let $\alpha=\alpha_{k-1}+\dots+\alpha_{0}<\ome^{\ome}$ be an ordinal described as a Cantor normal form.
Then, a set $X\subseteq_{\fin}\N$ is $\alpha$-large if and only if there is a partition $X=X_{0}\sqcup\dots\sqcup X_{k-1}$ such that $\max X_{i}<\min X_{i+1}$ and $X_{i}$ is $\alpha_{i}$-large.
 \item Let $n,m\in \N$.
If $X\subseteq_{\fin}\N$ is $(\ome^{n+m}+1)$-large and $\min X>3$, then there exists $Y\subseteq X$ such that $Y$ is $\ome^{n}$-large and $\ome^{m}$-sparse.
 \item If $X=Y_{0}\cup Y_{1}\subseteq_{\fin}\N$ is $\ome^{n}\cdot (4k)$-large and $\ome^{3}$-sparse, then $Y_{0}$ is $\ome^{n}\cdot k$-large or $Y_{1}$ is $\ome^{n}\cdot k$-large.
\end{enumerate}
\end{lemma}

For a precise calculation of a finite version of $\EM$, the grouping principle introduced in \cite{PY2018} is very useful.
\begin{definition}[grouping]
Let $\alpha, \beta<\ome^{\ome}$.
Let $X\subseteq\N$, and let $P:[X]^{2}\to \min X$ be a coloring.
A finite family (sequence) of finite sets $\langle F_{i}\subseteq X\mid i<l \rangle$ is said to be an \emph{$(\alpha,\beta)$-grouping for $P$} if
\begin{enumerate}
 \item $\A i<j<l\, \max F_{i}<\min F_{j}$,
 \item for any $i<l$, $F_{i}$ is $\alpha$-large,
% \item for any $H\subseteq_{\fin}\N$, if $H\cap F_{i}\neq\emptyset$ for any $i<l$, then $H$ is $\beta$-large, and,
 \item $\{\max F_{i}\mid i<l\}$ is $\beta$-large, and,
 \item $\A i<j<l\, \A x,x'\in F_{i}\,\A y,y'\in F_{j}\, P(x,y)=P(x',y')$.
\end{enumerate}

\end{definition}
We say that a set $X\subseteq \N$ \emph{admits $(\alpha,\beta)$-grouping for $k$-colors} if for any coloring $P:[X]^{2}\to k$, there exists an $(\alpha,\beta)$-grouping for $P$.
In \cite{KY-unpublished}, they considered colorings on $[X]^{2}$ using two colors, but here we consider colorings on $[X]^{2}$ using $\min X$ colors.
We will check that the following strengthening of Theorem~2.3 of \cite{KY-unpublished} still holds in our setting.
\begin{theorem}[$\II$]
\label{thm:grouping-main}
Let $n,k\in\N$.
If $X\subseteq_{\fin}\N$ is $\ome^{n+6k}$-large and $\ome^{3}$-sparse, then $X$ admits $(\omega^{n},\omega^{k})$-grouping for $\min X$-colors.
\end{theorem}

%To obtain a grouping, we need to stabilize the color between groups.
%We will first stabilize the color from below and above.
The proof is essentially the same as the original, but we need to upgrade some lemmas for the $\min X$ colors version.
\begin{lemma}[alteration of Lemma~2.4 of \cite{KY-unpublished}, $\II$]
\label{lem1:stabilize-coloring}
Let $X\subseteq_{\fin}\N$ be $\ome^{n+1}$-large and $\ome^{3}$-sparse, and let $c\in\N$ such that $4^{c^{2}}\le \min X$.
Then, we have the following.
\begin{enumerate}
 \item For any $\bar X\subseteq_{\fin}\N$ such that $|\bar X|\le c$ and $\max \bar X<\min X$ and for any coloring $P:[\bar X\cup X]^{2}\to c$, there exists $Y\subseteq X$ such that $Y$ is $\ome^{n}$-large and for any $x\in \bar X$ and $y,y'\in Y$, $P(x,y)=P(x,y')$.
 \item For any $\bar X\subseteq_{\fin}\N$ such that $|\bar X|\le c$ and $\max X<\min \bar X$ and for any coloring $P:[ X\cup \bar X]^{2}\to c$, there exists $Y\subseteq X$ such that $Y$ is $\ome^{n}$-large and for any $x\in \bar X$ and $y,y'\in Y$, $P(y,x)=P(y',x)$.
\end{enumerate}
\end{lemma}
\begin{proof}
We only show 1. (2 can be proved similarly.)
Since $X$ is $\ome^{n+1}$-large and $4^{c^{2}}<\min X$, $X\setminus\{\min X\}$ is $\ome^{n}\cdot 4^{c^{2}}$-large.
Put $Y_{0}=X\setminus\{\min X\}$.
Without loss of generality, we may assume that $\bar X=c$, so let $\{x_{i}: i<c\}$ be an enumeration of $\bar X$.
Construct a sequence $Y_{0}\supseteq Y_{1}\supseteq\dots\supseteq Y_{c^{2}}$ so that $Y_{i}$ is $\ome^{n}\cdot 4^{c^{2}-i}$-large.
% and $\A y,y'\in Y_{i+1}(P(x_{i},y)=P(x_{i},y'))$.
If $Y_{i}$ is given and $i=j_{1}c+j_{2}$ with $j_{1},j_{2}<c$, then $Y_{i+1}$ can be chosen to be $\{y\in Y_{i}: P(x_{j_{1}},y)=j_{2}\}$ or $\{y\in Y_{i}: P(x_{j_{1}},y)\neq j_{2}\}$ by Lemma~\ref{lem:from-KY}.3.
Then, $Y=Y_{c^{2}}$ is the desired set.
\end{proof}

%Next, we obtain a constant length grouping.
\begin{lemma}[alteration of Lemma~2.5 of \cite{KY-unpublished}, $\II$]
\label{lem2:ome-n-c-grouping}
Let $X\subseteq_{\fin}\N$ be $\ome^{n+3}$-large and $\ome^{3}$-sparse, and let $c\in\N$ such that ${c}\le \min X$.
Then, $X$ admits $(\ome^{n},c)$-grouping for $\min X$-colors.
\end{lemma}
\begin{proof}
The original proof works with the alternated Lemma~\ref{lem1:stabilize-coloring}.
\end{proof}
\begin{lemma}[alteration of Lemma~2.6 of \cite{KY-unpublished}, $\II$]
\label{lem3:ome-n-ome-grouping}
Let $X\subseteq_{\fin}\N$ be $\ome^{n+6}$-large and $\ome^{3}$-sparse.
Then, $X$ admits $(\ome^{n},\ome)$-grouping for $\min X$-colors.
\end{lemma}
\begin{proof}
The original proof works with the alternated Lemma~\ref{lem2:ome-n-c-grouping}.
\end{proof}

Finally we prove Theorem~\ref{thm:grouping-main} by using the previous lemma repeatedly.
\begin{proof}[\it Proof of Theorem~\ref{thm:grouping-main}.]
We will show by induction on $k$.
The case $k=0$ is trivial, and the case $k=1$ is Lemma~\ref{lem3:ome-n-ome-grouping}.
Assume $k\ge 2$, and let $X\subseteq_{\fin}\N$ be $\ome^{n+6k}$-large and $\ome^{3}$-sparse.
Fix a coloring $P:[X]^{2}\to \min X$.
Then, by Lemma~\ref{lem3:ome-n-ome-grouping}, there is an $(\ome^{n+6(k-1)},\omega)$-grouping $\langle Y_{i}:i\le \ell \rangle$ for $P$.
Here, $\{\max Y_{i}: i\le \ell\}$ is $\ome$-large, thus $\ell\ge \max Y_{0}$.
By the induction hypothesis, for each $1\le i\le \ell$, there is an $(\ome^{n},\ome^{k-1})$-grouping for $P$ $\langle Z^{i}_{j}\subseteq Y_{i}: j\le m_{i} \rangle$.
Since $\{\max Z^{i}_{j}: j\le m_{i}\}$ is $\ome^{k-1}$ for any $1\le i\le \ell$, the set $\{\max Y_{0}\}\cup\{\max Z^{i}_{j}: j\le m_{i},1\le i\le\ell\}$ is $\ome^{k}$-large.
Thus, $\langle Y_{0},Z^{1}_{0},\dots,Z^{1}_{m_{1}},\dots,Z^{\ell}_{0},\dots,Z^{\ell}_{m_{\ell}} \rangle$ is an $(\ome^{n},\ome^{k})$-grouping for $P$.
\end{proof}

We are now ready to prove Theorem~\ref{thm:EMinf-largeness-main}.
We will follow the idea of the proof of \cite[Lemma 7.2]{PY2018}.
The key idea here is that if $\langle X_{i}: i\le\ell \rangle$ is a grouping for $P$ and $P$ is fallow on any of $[X_{i}]^{2}$ and $[\{\max X_{i}: i\le\ell\}]^{2}$, then $P$ is fallow on $[\bigcup_{i\le\ell} X_{i}]^{2}$.
\begin{proof}[\it Proof of Theorem~\ref{thm:EMinf-largeness-main}.]
We first show 1.
By Lemma~\ref{lem:from-KY}.1, if $X$ is $\ome^{18n}$-large and $\min X>3$, then it is $\ome^{18(n-1)+7}+1$-large, hence one can take $X'\subseteq X$ so that $X'$ is $\ome^{18(n-1)+4}$-large and $\ome^{3}$-sparse.
So, it is enough to show that if $X$ is $\ome^{18(n-1)+4}$-large and $\ome^{3}$-sparse then it is $\EM$-$\omega^{n}$-large.
The case $n=1$ is Theorem~\ref{thm:EMinf-omega-large}.
Assume $n\ge 2$ and let $X\subseteq_{\fin}\N$ be $\ome^{18(n-1)+4}$-large.
Fix $P:[X]^{2}\to \min X$.
By Theorem~\ref{thm:grouping-main}, take an $(\ome^{18(n-2)+4}, \ome^{3})$-grouping $\langle Y_{i}:i\le \ell \rangle$ for $P$.
By applying Theorem~\ref{thm:EMinf-omega-large} for an $\ome^{3}$-large set $\{\max Y_{i}: i\le \ell\}$, take an $(\ome^{18(n-2)+4}, \ome)$-subgrouping $\langle Y_{i_{j}}:j\le \ell' \rangle$ such that $P$ is fallow on $[\{\max Y_{i_{j}}: j\le \ell'\}]^{2}$.
By the induction hypothesis, for each $j\le \ell'$, take $Z_{j}\subseteq Y_{i_{j}}$ such that $Z_{j}$ is $\ome^{n-1}$-large and $P$ is fallow on $[Z_{j}]^{2}$.
(Note that $\max Z_{0}\le \max Y_{i_{0}}\le \ell'$.)
Then, $H=\{\max Z_{0}\}\cup\bigcup_{1\le j\le \ell'}Z_{j}$ is $\ome^{n}$-large and $P$ is fallow on $[H]^{2}$.
This completes the proof of 1.

One may see 2 by induction on $n$.
The first condition for the density follows from 1.
The second condition for the density follows from Lemma~\ref{lem:from-KY}.1.
\end{proof}

\section{Questions}

While the result above places an upper bound on the strength of $\mathsf{fEM}_{<\infty}$, no corresponding lower bound is known.  More generally, it is unclear what the first order part of $\mathsf{fEM}_{<\infty}$ (and even $\mathsf{EM}$) is.  $\mathsf{EM}$ is known to imply $B\Sigma_2$ (\cite{MR2925280}), and the tree structure used in the forcing above (and in other arguments controlling the strength of $\mathsf{EM}$ \cite{MR3579121,Patey:2015bh}) is 
%essentially
related with a ``bounded monotone enumeration'' as introduced in \cite{MR3194495}.
Therefore it is possible that $\mathsf{EM}$ implies some variation of the principle $\mathsf{BME}$ which is not provable in $B\Sigma_2$ \cite{MR3518781}.
Note that $\mathsf{BME}$ itself is a $\tilde\Pi^{0}_{3}$-statement, so the original version is not provable from $\mathsf{WKL}_0+\mathsf{fEM}_{<\infty}$.
\begin{question}
  Does $\mathsf{RCA}_0+\mathsf{fEM}_{<\infty}$ or even $\mathsf{RCA}_0+\mathsf{EM}_{<\infty}$ have any first-order consequences over $B\Sigma_2$?
%
%  Does $\mathsf{RCA}_0+\mathsf{EM}$ imply $\mathsf{BME}$?  Is $\mathsf{RCA}_0+\mathsf{EM}$ $\Pi^1_1$-conservative over $B\Sigma_2+\mathsf{BME}$?
%
\end{question}

A related question is to consider how many instances of the ordinary $\mathsf{EM}$ it takes to find sets on which a coloring with more than two colors is transitive.  For instance, suppose we have a coloring $c:[\mathbb{M}]^2\rightarrow\{0,1,2\}$.  For each $i\in\{0,1,2\}$, we can consider the coloring $c_i:[\mathbb{M}]^2\rightarrow\{0,1\}$ given by $c_i(x,y)=\left\{\begin{array}{ll}1&\text{if }c(x,y)=i\\0&\text{otherwise}\end{array}\right.$.  By three applications of $\mathsf{EM}$, we obtain a set $S$ transitive under all three colorings $c_0,c_1,c_2$ simultaneously, and therefore transitive under $c$.

With similar principles---like Ramsey's theorem for pairs itself \cite{MR3430365}---one needs fewer iterations to increase the number of colors.

\begin{question}
  Let $c:[\mathbb{M}]^2\rightarrow[0,n)$ be given.  Is it possible to find an infinite set $S$ so that $c\upharpoonright S$ is transitive using fewer than $n$ instances of $\mathsf{EM}$?
\end{question}

Finally, one could consider a version of $\mathsf{EM}$ for infinitely many colors at once:
\begin{definition}
  $\mathsf{EM}_\infty$ holds if whenever $c:[\mathbb{M}]^2\rightarrow\mathbb{M}$ is a coloring, there is an infinite set $S$ so that $c\upharpoonright [S]^2$ is transitive.
\end{definition}

\begin{lemma}
  $\mathsf{ACA}_0$ implies $\mathsf{EM}_\infty$.
\end{lemma}
\begin{proof}
Let $c:[\mathbb{M}]^2\rightarrow\mathbb{M}$ be given.  We use the well-known fact that Ramsey's theorem for triples is provable in $\mathsf{ACA}_0$: define 
\[c'(x,y,z)=\left\{\begin{array}{ll}
1&\text{if }c\text{ is transitive on the set }\{x,y,z\}\\
0&\text{otherwise}
\end{array}\right..\]
Note that $c$ is transitive on $\{x,y,z\}$ with $x<y<z$ if either $c(x,y)\neq c(y,z)$ or $c(x,y)=c(y,z)=c(x,z)$.

By Ramsey's theorem for triples, we obtain an infinite set $S$ so that $c$ is homogeneous on $[S]^3$.  If $c$ were homogeneously $0$, we could find four elements $x,y,z,w$ so that $c$ fails to be transitive on every triple.  Then we must have $c(x,y)=c(y,z)\neq c(x,z)$, and $c(y,z)=c(z,w)\neq c(y,w)$.  But then $c(x,y)\neq c(y,w)$, so $c$ is transitive on $\{x,y,w\}$, contradicting the homogeneity of $c$.
\end{proof}

\begin{question}
  Does $\mathsf{EM}_\infty$ imply $\mathsf{ACA}_0$?
\end{question}

\printbibliography

\end{document}